\documentclass{amsart}

\usepackage{pinlabel} 
\usepackage{amsmath}
\usepackage{amsfonts}
\usepackage{amssymb}
\usepackage{mathtools}
\usepackage{amscd}
\usepackage{enumerate}
\usepackage[all,cmtip]{xy}
\usepackage{amsthm}
\usepackage{comment}
\usepackage{anysize}
\usepackage{epsfig}
\usepackage{hyperref}
\usepackage[colorinlistoftodos]{todonotes}
\usepackage[a4paper, total={6in, 8in}]{geometry}
\usepackage[utf8]{inputenc}
\usepackage{xcolor}
\usepackage[capitalise]{cleveref}
\usepackage{tikz-cd}
\usepackage{float}
\restylefloat{table}

\usepackage{array}
\newcolumntype{P}[1]{>{\centering\arraybackslash}p{#1}}

\makeatletter
\providecommand\@dotsep{5}
\def\listtodoname{List of Todos}
\def\listoftodos{\@starttoc{tdo}\listtodoname}
\makeatother

\newtheorem{theorem}{Theorem}[section]
\newtheorem{proposition}[theorem]{Proposition}
\newtheorem{corollary}[theorem]{Corollary}
\newtheorem{lemma}[theorem]{Lemma}

  \theoremstyle{definition}
\newtheorem{definition}[theorem]{Definition}

\newtheorem{remark}[theorem]{Remark}

\newcommand{\dbE}{\mathbb{E}}

\newcommand{\dbH}{{\mathbb H}}

\newcommand{\dbR}{\mathbb{R}}

\newcommand{\dbZ}{\mathbb{Z}}

\newcommand{\calA}{{\mathcal A}}

\newcommand{\calF}{{\mathcal F}}
\newcommand{\calG}{{\mathcal G}}
\newcommand{\calH}{{\mathcal H}}

\newcommand{\calK}{{\mathcal K}}

\newcommand{\vcyc}{V\text{\tiny{\textit{CYC}}}}

\newcommand{\fin}{F\text{\tiny{\textit{IN}}}}

\newcommand{\nbeq}{\begin{equation}}
\newcommand{\neeq}{\end{equation}}
\newcommand{\beq}{\begin{equation*}}
\newcommand{\eeq}{\end{equation*}}

\newcommand{\gldosz}{GL_2(\dbZ)}

\newcommand{\mbr}{\mathbb{R}}

\newcommand{\mbe}{\mathbb{E}}

\newcommand{\mbs}{\mathbb{S}}

\newcommand{\pslt}{\widetilde{\mathrm{PSL}}_2(\mbr)}
\newcommand{\hyp}{\mathbb{H}^2}
\newcommand{\hypp}{\mathbb{H}^3}

\newcommand{\nil}{\mathrm{Nil}}
\newcommand{\sol}{\mathrm{Sol}}

\newcommand{\iso}{\mathrm{Iso}}

\newcommand{\hxr}{\mathbb{H}^2\times\mathbb{E}}

\newcommand{\aut}{\mathrm{Aut}}

\newcommand{\isom}{\mathrm{Isom}}

\DeclareMathOperator{\cyl}{Cyl}

\DeclareMathOperator{\cd}{cd}

\DeclareMathOperator{\Z}{\mathbb{Z}}
\DeclareMathOperator{\E}{\mathbb{E}}
\DeclareMathOperator{\R}{\mathbb{R}}
\DeclareMathOperator{\I}{I}

\DeclareMathOperator{\maH}{\mathbb{H}}
\DeclareMathOperator{\gd}{gd}

\DeclareMathOperator{\Aut}{Aut}

\DeclareMathOperator{\stab}{Stab}

\DeclareMathOperator{\SUB}{SUB}
\DeclareMathOperator{\fix}{Fix}

\DeclareMathOperator{\ef2}{E_{\mathcal{F}_2}}

\begin{document}

\title[Classifying spaces for orientable $3$-manifold groups]{Classifying spaces for the family of virtually abelian subgroups of orientable $3$-manifold groups}

\author{Porfirio L. León Álvarez}
\address{Instituto de Matemáticas, Universidad Nacional Autónoma de México. Oaxaca de Juárez, Oaxaca, México 68000}
\email{porfirio.leon@im.unam.mx}

\author{Luis Jorge S\'anchez Salda\~na}
\address{Departamento de Matemáticas,
Facultad de Ciencias, Universidad Nacional Autónoma de México, Circuito Exterior S/N, Cd. Universitaria, Colonia Copilco el Bajo, Delegación Coyoacán, 04510, México D.F., Mexico}
\email{luisjorge@ciencias.unam.mx}


\date{}



\begin{abstract}
For a group $G$, let $\mathcal{F}_{n}$ be the family  of all  the subgroups of  $G$ containing a subgroup    of finite index isomorphic to  $\mathbb{Z}^{r}$ for some $r=0,1,2,\dots ,n$. Joecken, Lafont and Sánchez Saldaña computed the $\calF_1$-geometric dimension of 3-manifold groups. As a natural extension of the aforementioned result, the goal of this  article  is to compute the  $\calF_n$-geometric dimension   of 3-manifold groups  for all $n\ge 2$. 
\end{abstract}

\maketitle

\section{Introduction}
Given a group $\Gamma$, we say a collection $\calF$ of subgroups of $\Gamma$ is a \emph{family} if it is non-empty, closed under conjugation and under taking subgroups. Fix a group $\Gamma$ and a family $\calF$ of subgroups of $\Gamma$. We say that a $\Gamma$-CW-complex $X$ is a \emph{model for the classifying space} $E_{\calF}\Gamma$ if every isotropy  group of $X$ belongs to the family $\calF$ and the fixed point set $X^H$ is contractible whenever  $H$ belongs to $\calF$. It can be shown that a model for the classifying space   $E_{\calF}\Gamma$ always exists and it is unique up to $\Gamma$-homotopy equivalence. We define the \emph{$\calF$-geometric dimension} of $\Gamma$   as
$$\gd_{\calF}(\Gamma)=\min\{n\in \mathbb{N}| \text{ there is a model for } E_{\calF}\Gamma \text{ of dimension } n \}.$$
Examples of  families of subgroups are: the family that only consists of the trivial subgroup $\{1\}$, and the family $\fin$ of finite subgroups of $\Gamma$. The $\calF$-geometric dimension  has been widely studied in the last decades, and the present paper contributes to this topic.

Let $n\geq0$ be an integer. A group is said to be \emph{virtually $\mathbb{Z}^n$} if it contains a subgroup of finite index isomorphic to $\mathbb{Z}^n$. Define the family \[\calF_n=\{H\leq\Gamma| H \text{ is virtually } \mathbb{Z}^r \text{ for some 
} 0\le r \le n \}.\] The  families $\{1\}$, $\calF_0=\fin$ and $\calF_1=\vcyc$  are relevant due to its connection with the Farrell-Jones and Baum-Connes isomorphism conjectures.  The families $\calF_n$  have been recently studied by
several people, see for instance \cite{Nucinkis:Moreno:Pasini,prytua2018bredon,HP20,SS20}.
 
 In  \cite{Luis:Lafon:Joeken} Joecken, Lafont and Sánchez Saldaña computed the $\calF_1$-geometric dimension of  3-manifold groups, that is fundamental groups of connected, closed, orientable $3$-manifolds. The main goal of the present article is to set up a  natural extension of this result: we explicitly compute the $\calF_k$-geometric dimension of 3-manifold groups  for all $k \geq 2$.  Actually for a $3$-manifold group we have $\calF_3=\calF_k$ for all $k\geq 4$, see \cref{coro:Zn:subgroups}, thus our computations only deal with the families $\calF_2$ and $\calF_3$.
 
Throughout the present article we always consider orientable manifolds. To state our first main theorem we use the prime decomposition of a 3-manifold. Recall that a 3-manifold $P$ is \emph{prime} if $P=P_1\#P_2$ implies $P_1$ or $P_2$ is the 3-sphere. It is well-known that every connected, closed 3-manifold has a unique decomposition as connected sum of prime 3-manifolds, see \cref{thm:prime:decomposition}.
 
\begin{theorem}\label{thm:principal:1}
Let $M$ be a connected, closed, and oriented 3-manifold.
Let $P_1,P_2,\dots,P_r$ be the pieces in the prime decomposition of $M$. Denote  $\Gamma=\pi_1(M)$ and $\Gamma_i=\pi_1(P_i)$. Then, for all $k\geq 2$,  \[
\gd_{\calF_k}(\Gamma)=\begin{cases}
  0 & \text{if $M=\mathbb RP^3\# \mathbb RP^3$,} \\
  2 & \text{if $r\geq 2$,  $\Gamma_i\in \calF_k$ for all $1\leq i \leq r$,} \\ &\text{ and $\Gamma$ is not virtually cyclic,}\\
  \max\{\gd_{\calF_k}(\Gamma_i)|1\le i\le r\} & \text{otherwise.}
\end{cases}\]
Moreover
\begin{itemize}
    \item $\Gamma_i\in \calF_2$ if and only if $P_i$ is modelled on $\mbs^3$ or $\mbs^2\times \dbE$ or $P_i=\mathbb RP^3\# \mathbb RP^3$.
    \item $\Gamma_i\in \calF_3$ if and only if $\Gamma_i\in \calF_2$ or $\Gamma_i$ is modelled on $\dbE^3$. 
\end{itemize}
\end{theorem}

In the light of the previous theorem, the task of computing the $\calF_k$-geometric dimension of a 3-manifold group, it is reduced to computing the $\calF_k$-geometric dimension of  fundamental groups of  \emph{prime} 3-manifolds. If $M$ is a prime manifold, then we can cut-off along certain embedded tori in such a way that the resulting connected components are either hyperbolic or Seifert fiber 3-manifolds. This is the famous JSJ-decomposition after the work of Perelman, see \cref{JSJ:decomposition}.

\begin{theorem}\label{geometric:dimension:prime}
Let $M$ be a connected, closed, and oriented prime 3-manifold.
Let $N_1,N_2,\dots,N_r$, be the pieces in the minimal JSJ-decomposition  of $M$. Denote $\Gamma=\pi_1(M)$ and $\Gamma_i=\pi_1(N_i)$. If $k\ge 2$, then
\[\gd_{\calF_k}(\Gamma)=\begin{cases} 
2 & \text{if $M$ is modelled on } \mathrm{Sol},\\
\max\{\gd_{\calF_k}(\Gamma_i)|1\le i\le r\} & \text{otherwise.}
\end{cases}\]
\end{theorem}

From the previous theorem, it is clear that our next task is to compute the  $\calF_k$-geometric dimension of the fundamental group of all possible JSJ-pieces, that is, the fundamental groups of hyperbolic and Seifert fibered 3-manifolds. We accomplished this and we summarize our results in \cref{summary:dimention:pieces:jsj:Fk}. In the $\gd_{\calF_2}$ column we reference the theorems where the computations were carried out. The last column is justified in \cref{prop:F3:equal:Fk}.

\begin{table}[H]
\centering
\begin{tabular}{|P{5.5cm}|P{4cm}|P{3.5cm}|}
 \hline
 \textbf{Type of manifold}  &\textbf{$\gd_{\calF_2}(\Gamma)$/proved in} &\textbf{$\gd_{\calF_k}(\Gamma)$ with $k\ge 3$ }\\
 \hline
   Hyperbolic with empty boundary     & 3, \cref{thm:hyperbolic:case} & 3\\
  \hline 
 Hyperbolic with non-empty boundary  & 3, \cref{thm:hyperbolic:case}& 3\\
 \hline
  Seifert fiber with base orbifold  $B$ which is either bad or modelled on $\mbs^2$  &0, \cref{thm:seifer:base:spherical} & 0\\
 \hline
 Seifert fiber with base orbifold $B$  modelled  on $\mathbb{H}^2$, and boundary empty or non-empty  & 2, \cref{thm:seifer:base:hyperbolic} & 2\\
 \hline
  Seifert fiber modelled on $\mathbb{E}^3$ with empty boundary and base orbifold $B$  modelled  on $\mathbb{E}^2$   & 5, \cref{thm:seifer:base:flat} & 0 \\
  \hline
   Seifert fiber modelled on $\nil$ with empty boundary and base orbifold $B$  modelled  on $\mathbb{E}^2$ & 3, \cref{thm:seifer:base:flat} & 3 \\
  \hline
   Seifert fibered  with non-empty and base orbifold  $B$  modelled on $\mathbb{E}^2$ & 0, \cref{thm:seifer:base:flat:bondary} & 0 \\
 \hline
\end{tabular}
\caption{   $\calF_k$-geometric dimension  of the pieces JSJ.}
\label{summary:dimention:pieces:jsj:Fk}
\end{table}

\subsection*{Outline of the paper.} In \cref{sec-background} we set up the preliminaries such as the definition and some properties of the $\calF$-geometric dimension, Bass-Serre theory, 3-manifold theory, and some push-out type constructions for classifying spaces for families. In \cref{section:hyperbolic"}, \cref{section:seifert}, and \cref{section:sol} we compute the $\calF_2$-geometric dimension of hyperbolic, Seifert fiber, and $\sol$-manifolds respectively. In \cref{section:acylindricity} we recall the notion of acylindrical splitting for the fundamental group of a graph of groups, then we prove two results that will be useful to use the computations of the previous sections to prove out our main theorems. We can think of the results in \cref{section:acylindricity} as the tools that help us to glue together the classifying spaces of the building pieces of a 3-manifold (the prime and JSJ pieces). Finally in \cref{section:final} we prove \cref{thm:principal:1} and \cref{geometric:dimension:prime}.

\subsection*{Acknowledgements}Both authors were supported by grant PAPIIT-IA101221. P.L.L.A was supported by a doctoral scholarship of the Mexican Council of Science and Technology (CONACyT). Both authors thank Rita Jiménez Rolland for comments on a draft of the present article. We are grateful to the anonymous referee, his/her comments  helped to improve the readability of the paper.


\section{Preliminaries}\label{sec-background}

\subsection{Geometric and cohomological dimension for families}

Given a group $\Gamma$, we say a collection $\calF$ of subgroups of $\Gamma$ is a \emph{family} if it is non-empty, closed under conjugation and under taking subgroups. Fix a group $\Gamma$ and a family $\calF$ of subgroups of $\Gamma$. We say that a $\Gamma$-CW-complex $X$ is a \emph{model for the classifying space} $E_{\calF}\Gamma$ if every isotropy  group of $X$ belongs to the family $\calF$ and the fixed point set $X^H$ is contractible whenever  $H$ belongs to $\calF$. It can be shown that a model for the classifying space   $E_{\calF}\Gamma$ always exists and it is unique up to $\Gamma$-homotopy equivalence. We define the \emph{$\calF$-geometric dimension} of $\Gamma$   as
$$\gd_{\calF}(\Gamma)=\min\{n\in \mathbb{N}| \text{ there is a model for } E_{\calF}\Gamma \text{ of dimension } n \}.$$
 
 The \emph{orbit category} $\mathcal{O}_{\calF}\Gamma$ is the category whose objects are $G$-homogenous spaces $\Gamma/H$ with $H\in \calF$ and morphisms are $\Gamma$-functions. The \emph{category of Bredon modules}  is the category whose objects are contravariant functors $M\colon \mathcal{O}_{\calF}\Gamma \to Ab$ from the orbit category to the category of  abelian group, and morphisms are natural transformation $f\colon M\to N$. This is an abelian category  with enough proyectives. The constant Bredon module  $\underline{\mathbb{Z}}\colon \mathcal{O}_{\calF}\Gamma \to Ab$ is defined in objects by $\underline{\mathbb{Z}}(\Gamma/H)=\mathbb{Z}$ and in morphisms by $\underline{\mathbb{Z}}(\varphi)=id_{\mathbb{Z}}$. We define the \emph{$\calF$-cohomological dimension} of $\Gamma$ as 
$$\cd_{\calF}(\Gamma)=\min\{n\in \mathbb{N}| \text{ there is a projective resolution of } \underline{\mathbb{Z}} \text{ of dimension } n \}.$$

The proof of the following proposition is implicit in \cite[Proof~of~theorem~3.1]{Lu00} and \cite[Proposition~5.1]{luck:Weiermann}, see also \cite[Theorem~2.3]{MPSS20}. 

\begin{proposition}\label{prop:haefliger}
Let $G$ be a group. Let $\calF$ and $\calG$ be families of subgroups of $G$ such that $\calF\subseteq \calG$. If $X$ is a model for $E_\calG G$, then
\[\gd_{\calF}(G)\leq \max\{\gd_{\calF\cap G_\sigma}(G_{\sigma})+\dim(\sigma)| \ 
\sigma \text{ is a cell of } \ X \}.\] 
\end{proposition}

    \begin{lemma}\label{lemma:low:dimensions}
Let $n\geq 1$ and let $G$ be a finitely generated group. Assume that for all $1\leq r\leq n$, $G$  is not virtually $\dbZ^r$. Then $\gd_{\calF_n}(G)\geq 2$.
\end{lemma}
\begin{proof}
It is easy to see that $\gd_{\calF_n}(G)=0$ if and only if $G\in \calF_n$. Thus by hypothesis $\gd_{\calF_n}(G)\neq 0$. Now assume $\gd_{\calF_n}(G)=1$, then $G$ acts on a tree with stabilizers in $\calF_n$. In particular every element $g\in G$ has a fixed point. Since $G$ is finitely generated, by \cite[Corollary 3, p.65]{Serre:trees}, $G$ has a global fixed point in $X$, but this contradicts our hypothesis that $G\notin \calF_n$.
\end{proof}
\subsection{Bass-Serre theory}
In this section we remember some notation and results of the Bass-Serre theory that we will use later.

A \emph{graph} $Y$(in the sense of Serre) consists of a set of vertices $V(Y)$, a set of edges $E(Y)$ and two functions $E(Y)\to V(Y)\times V(Y)$, $y\mapsto(o(Y),t(Y))$ and $E(Y)\to E(Y)$, $y\mapsto \overline{y}$ satisfying $\Bar{\Bar{y}}=y$ and $o(y)=t(\overline{y})$ for all $y\in E(Y)$. The \emph{geometric realization} of $Y$ is the quotient space, $V(Y) \bigsqcup(E(Y)\times I)/\sim$, where $V(Y)$ and $E(Y)$ have the discrete topology, and the equivalence relation in  $V(Y) \bigsqcup(E(Y)\times I)$ is given as follows: for every $y\in E(Y)$ and $t\in I$, $(y,t)\sim (\bar{y},1-t)$ , $(y,0)\sim o(y)$ and $(y,1)\sim t(y)$.

A \emph{graph of groups} $\mathbf{Y}$ consists of  a graph $Y$, a group $Y_P$ for each $P\in V(Y)$, and a group $Y_y$ for each $y\in E(Y)$, together with  monomorphisms  $\varphi_{y}\colon Y_y \to Y_{t(y)}$. One requires in addition $Y_{\Bar{y}}=Y_y$.
Suppose that the group $G$ acts without inversions on the graph $Y$, i.e. for every $g\in G$ and $y\in E(Y)$, we have $gy\neq \bar{y}$. Then we have an induced graph of groups with underling graph $Y/G$ by associating to each vertex (resp. edge) the isotropy group of a preimage under the quotient map $Y\to Y/G$.

Given a graph of groups $\mathbf{Y}$, one the classic theorems of Bass-Serre theory provides the existence of a group $G=\pi_1(\mathbf{Y})$, called the \emph{fundamental group of the graph of groups} $\mathbf{Y}$ and the tree $T$(a graph with no cycles), called the \emph{Bass-Serre tree} of $\mathbf{Y}$, such that $G$ acts on $T$ without inversions, and the induced graph of groups is isomorphic to $\mathbf{Y}$. The identification $G=\pi_1(\mathbf{Y})$ is called a splitting of $G$.

As a direct consequence of \cite[Lemma 1.1]{Dunwoody:Sageev} we get the following result which will be useful later on.
\begin{proposition}\label{finitely:generated:group:fix:edge:or:act:path}
Let $H$ be a group virtually $\Z^n$ acting  on a tree $T$. Then exactly one of the  following happens:
\begin{enumerate}[a)]
    \item $H$ fixes  a vertex of $T$.
    \item $H$ acts  co-compactly in a unique geodesic line $\gamma$ of $T$.  
\end{enumerate}
\end{proposition}

\subsection{3-manifolds and their decompositions} In this section we revisit some results of 3-manifolds that we use later. For more details see \cite{Sc83} and \cite{matthias:stefan:henry}.

A Seifert fibered space is a 3-manifold $M$ with a decomposition of $M$ into a disjoint union of circles, called fibres, such that each circle has a tubular neighbourhood in $M$ that is isomorphic to a fibered solid torus or Klein bottle.  If we collapse each  of these circles we obtain a surface $B$ that has a natural orbifold structure, we call $B$ the \emph{base orbifold} of $M$. Such an orbifold $B$ has its \emph{orbifold fundamental group} $\pi_1^{orb}(B)$, which is not necessarily the fundamental group of the underlying topological space, but it is related to the fundamental group of $M$ via the following lemma.

\begin{lemma}\cite[lemma 3.2]{Sc83}\label{sequence:exact:groups:fundamental:seifert}
Let $M$ be a Seifert fiber space with base orbifold $B$. Let $\Gamma$ be the fundamental group of $M$. Then there is an exact sequence 
 $$1\to K\to \Gamma\to \pi_1^{orb}(B)\to 1$$
where $K$ denotes the cyclic subgroup of $\Gamma$ generated by a regular fiber.
The group $K$ is infinite except in cases where $M$ is covered by $\mbs^3$.  \end{lemma}

A 2-orbifold $B$ is of exactly one of the following types depending on the structure of its universal orbifold covering: bad, spheric, hyperbolic, or flat. In this paper we will divide the computation of the $\calF_k$-geometric dimension of a Seifert manifold $M$ depending on the type of its base orbifold $B$.

The following is a well-known theorem of Kneser (existence) \cite{Kneser1929} and Milnor (uniqueness)\cite{MR142125}, see \cite[Theorem~2.1.2]{John:W:Morgan}.

\begin{theorem}[Prime decomposition]\label{thm:prime:decomposition}
Let $M$ be a closed, oriented 3-manifold. Then $P_1\#\cdots \#P_n$ where each $P_i$ es prime. Furthermore, this decomposition is unique up to order and homeomorphism. 
\end{theorem}

Another well-known result we will need is the Jaco-Shalen \cite{MR539411} and Johanson \cite{MR551744} decomposition, after Perelman's work, see \cite[Theorem~1.7.6]{matthias:stefan:henry}.
\begin{theorem}[JSJ-decomposition]\label{JSJ:decomposition}
For a closed, prime, oriented 3-manifold $M$ there is a collection $T\subseteq M$ of disjoint incompressible tori, i.e. two sided property embedded and $\pi_1$-injective, such that each component of $M-T$ is either a hyperbolic or a Seifert fibered manifold. A minimal such collection $T$ is unique up to isotopy.   
\end{theorem}

It is a consequence of the uniformization theorem that compact surfaces
(2-manifolds) admit Riemannian metrics with constant curvature; that is, compact
surfaces admit geometric structures modelled on $\mbs^2$, $\mbe^2$, or $\hyp$,  see \cite[Theorem~1.1.1]{John:W:Morgan}. 
In dimension three, we are not guaranteed constant curvature.  Thurston
demonstrated that there are eight $3$-dimensional maximal geometries up to
equivalence (\cite[Theorem 5.1]{Sc83}): $\mbs^3$, $\mbe^3$, $\hypp$,
$\mbs^2\times\mbe$, $\hxr$, $\pslt$, $\nil$, and $\sol$. A manifold
$M$ is called \emph{geometric} if there is a geometry $X$ and discrete subgroup
$\Gamma\leq\isom(X)$ with free $\Gamma$-action on $X$ such that $M$ is
diffeomorphic to the quotient $X/\Gamma$; we also say that $M$ \emph{admits a
geometric structure} modelled on $X$.  Similarly, a manifold with nonempty
boundary is geometric if its interior is geometric. It is worth saying that a 3-manifold is Seifert fiber if and only if it admits a geometry modelled on  $\mbs^3$, $\mbe^3$,
$\mbs^2\times\mbe$, $\hxr$, $\pslt$, or $\nil$, see for example \cite[Theorem 1.8.1]{matthias:stefan:henry} or \cite[Theorem 5.3]{Sc83}, we will use this fact all along the paper.

\subsection{Push-out constructions for classifying spaces} In this section we revisit some results that help us to construct classifying spaces using homotopy push-outs of other classifying spaces. 

\begin{definition}\label{ad1}
Let $\Gamma$ be a finitely generated  group , and $\mathcal{F}\subset \mathcal{F^{\prime}}$ a pair of families of  subgroups of $\Gamma$.  We say a collection $\mathcal{A}=\{A_{\alpha}\}_{\alpha\in I}$ of subgroups of  $\Gamma$ is adapted to the pair  $(\mathcal{F}, \mathcal{F^{\prime}})$ if the following conditions  hold:
\begin{enumerate}[a)]
    \item For all $A,B \in \mathcal{A}$, either $A=B$ or $A\cap B\in \mathcal{F}$;
    \item The collection $\mathcal{A}$ is closed under conjugation;
    \item Every  $A\in \mathcal{A}$ is self normalizing, i.e. $N_{\Gamma}(A)=A$;
    \item For all $A\in \mathcal{F}^{\prime}- \mathcal{F}$, there is $B\in \mathcal{A}$ such that $A\le B.$
\end{enumerate}
\end{definition}

\begin{theorem}\label{construction:pusout}\cite[P.~302]{LO09}
Let $\mathcal{F}\subset \mathcal{F^{\prime}}$ families of subgroups of $\Gamma$. Assume that the  collection $\mathcal{A}=\{H_{\alpha}\}_{\alpha\in I}$ is adapted to the pair  $(\mathcal{F}, \mathcal{F^{\prime}})$. Let $\mathcal{H}$ a complete set of reprensentatives of the conjugacy classes within  $\mathcal{A}$, and consider the cellular $\Gamma$-push-out  

\[
\begin{tikzpicture}
  \matrix (m) [matrix of math nodes,row sep=3em,column sep=4em,minimum width=2em]
  {
     \displaystyle\bigsqcup_{H\in \mathcal{H}}\Gamma\times_{H}E_{\mathcal{F}}H & E_{\mathcal{F}}\Gamma \\
     \displaystyle\bigsqcup_{H\in \mathcal{H}}\Gamma\times_{H}E_{\mathcal{F^{\prime}}}H & X \\};
  \path[-stealth]
    (m-1-1) edge node [left] {$f$} (m-2-1) (m-1-1.east|-m-1-2) edge  node [above] {$g$} (m-1-2)
    (m-2-1.east|-m-2-2) edge node [below] {$\varphi$} (m-2-2)
    (m-1-2) edge node [right] {$h$} (m-2-2);
\end{tikzpicture}
\]

Then $X$ is a model for  $E_{\mathcal{F^{\prime}}}\Gamma$. In the above  $\Gamma$-push-out  we require  either  (1) $f$ is the disjoint union of  cellular $H$-maps, and $g$ be a inclusion of  $\Gamma$-CW-complexes, or (2) $f$ is the disjoint union of inclusions of  $H$-CW-complexes, and $g$ is a cellular $\Gamma$-map. 
\end{theorem}

Let $\varphi \colon \Gamma\to \Gamma_0$ be a surjective homomorphism. If $H$ is a subgroup of $\Gamma_0$ we will denote by $\tilde H$ the subgroup $\varphi^{-1}(H)$ of $\Gamma$. Given a family $\calF$ of subgroups of $\Gamma_0$ we denote by $\tilde \calF$ the smallest family of $\Gamma$ that contains the set $\{\tilde H\colon H\in\calF\}$.

\begin{theorem}\cite[Theorem 4.5.]{Luis:Lafon:Joeken}\label{construction:pushout:pullbak}
Let $\mathcal{F}$ be a family of subgroups of the  finitely generated discrete group  $\Gamma$. Let $\varphi\colon \Gamma\to \Gamma_{0}$ be  a surjective homomorphism. Let $\mathcal{F}_{0}\subseteq\mathcal{F}_{0}^{\prime}$ be a  nested pair of families of subgroups  of $\Gamma_{0}$ satisfying $\tilde{\mathcal{F}_{0}}\subseteq \mathcal{F}\subseteq\tilde{\mathcal{F}_{0}^{\prime}}$, and let  $\mathcal{A}=\{A_{\alpha}\}_{\alpha\in I}$  be a collection adapted to the pair $\mathcal{F}_{0}\subseteq\mathcal{F}_{0}^{\prime}$. Let  $\mathcal{H}$ be a complete set of representatives of the conjugacy classes in  $\tilde{\mathcal{A}}=\{\varphi^{-1}(A_{\alpha})\}_{\alpha\in I}$, and consider the following cellular $\Gamma$-push-out

\begin{equation*}
\begin{tikzcd}
\displaystyle\bigsqcup_{\tilde{H}\in \mathcal{H}}\Gamma\times_{\tilde{H}}E_{\mathcal{F}_{0}}H \arrow[r, "g"] \arrow[d, "f"]
& E_{\calF_0}\Gamma_0  \arrow[d, "h" ] \\
\displaystyle\bigsqcup_{H\in \mathcal{H}}\Gamma\times_{\tilde{H}}E_{\mathcal{F}}\tilde{H} \arrow[r, "\varphi" ]
& X
\end{tikzcd}
\end{equation*}
Then  $X$ is a model for  $E_{\mathcal{F}}\Gamma$. In the  $\Gamma$-push-out above  we require either  (1) $f$ is the disjoint union  of cellular   $\tilde{H}$-maps, and $g$ is an inclusion of  $\Gamma$-CW-complexes, or (2) $f$ is the disjoint union of inclusions of   $\tilde{H}$-CW-complexes, and $g$ is the cellular   $\Gamma$-map.
\end{theorem}

\begin{remark}\label{conditions:exist:pushout} Conditions (1) and (2) at the end of the statements of \cref{construction:pusout} and \cref{construction:pushout:pullbak}  are required so that the $\Gamma$-push-outs in both statements are actually \emph{homotopy} $\Gamma$-push-outs. It is worth saying that both conditions can be always achieved using a simple \emph{cylinder replacement trick}. Let us explain how the map $f$ is defined in the context of \cref{construction:pushout:pullbak}. For $\tilde H\in \mathcal H$, we have a unique $H$-map $E_{\calF_0} H \to E_{\calF_0}\Gamma_0$, where the codomain is an $\tilde H$-CW-complex by restriction. By a standard argument, we get an induced $\Gamma$-map $\Gamma\times_{\tilde H} E_{\calF_0}H\to E_{\calF_0}\Gamma_0$, and $f$ is the union of these maps running over all elements of $\mathcal H$. The construction of $g$ is analogous. Both maps $g$ and $f$ can be chosen in such a way that they satisfy conditions (1) and (2). By the equivariant cellular approximation theorem $f$ and $g$ can be taken to be cellular $\Gamma$-maps. Also the  $\Gamma$-map $g$ (resp. $f$) can be replaced, if necessary, by the  inclusion $ i\colon \bigsqcup\Gamma\times_{\tilde H} E_{\calF_0}H \to \cyl (g)$, where $\cyl(g)$ is the mapping cylinder of $g$ which equivariantly deformation retracts onto the original $E_{\calF_0} \Gamma_0$ and therefore is still a model for the same classifying space. Also note that dimension of $\cyl(g)$ is   $\max \{ \dim(E_{\calF_0}H)+1, \dim(E_{\calF_0}\Gamma_0) \}$, for more details see \cite[Remark 2.5]{luck:Weiermann}. This cylinder construction will be used repeatedly (if necessary) all throughout the paper.
\end{remark}

\section{The hyperbolic case}\label{section:hyperbolic"}

In this one theorem section we compute the $\calF_2$-geometric dimension of hyperbolic manifolds with or without boundary. 

\begin{theorem}\label{thm:hyperbolic:case}
Let $M$ be a hyperbolic 3-manifold  of finite volume (possibly with non-empty boundary) and $\Gamma=\pi_{1}(M,x_{0})$. Then $\gd_{\mathcal{F}_{2}}(\Gamma)= 3$.
 \end{theorem}
 
\begin{proof}
We have two cases depending on whether the boundary of $M$  is empty  or not. First suppose that  $M$ has empty boundary.
 From  \cite[Corollary 4.6.]{Sc83} $\Gamma$ cannot contain a subgroup isomorphic to $\mathbb{Z}^{2}$, in consequence $\mathcal{F}_{1}=\mathcal{F}_{k}$ for all $k\ge 1.$  Thus, by  \cite[Proposition 6.1.]{Luis:Lafon:Joeken}, $\gd_{\mathcal{F}_{1}}(\Gamma)=3$, and  the statement follows.

 Now suppose that $M$ has non-empty boundary. Let us point out first two observations:
 
 \begin{enumerate}
    \item The first one is that $\Gamma$ is torsion-free. In fact, since $M$ is hyperbolic, and in particular aspherical, then the universal covering $\dbH^3$ of $M$ is a model for $E\Gamma$, and therefore $\Gamma$ has finite cohomological dimension. It follows from \cite[Corollary 2.5]{Brown} that $\Gamma$ must be torsion-free.
    
    \item  The second observation is that every infinite virtually cyclic subgroup of $\Gamma$ is isomorphic to $\dbZ$, this follows from the first observation and the classification of virtually cyclic groups, see for example \cite[Proposition 4]{Pineda:Leary}.
\end{enumerate}

 Consider the collection $\mathcal{B}$ of subgroups of $\Gamma$ consisting of
 \begin{itemize}
     \item All conjugates of the fundamental groups of the cusps of $M$. All these groups are isomomorphic to $\Z^2$ since every cusp is homemorphic to the product of a torus and an interval.
     
     \item All maximal infinite virtually cyclic of subgroups of $\Gamma$ that are not subconjugate to the fundamental group of a cusp. All these groups are isomorphic to $\Z$ by observation (2) above.
 \end{itemize}

Note that  $\Gamma$ is hyperbolic relative to the collection given by the fundamental groups of the cusps of $M$, see for example \cite[Example I, p.4]{Osin}. By \cite[Theorem 2.6]{lafont:ortiz}, the collection $\mathcal{B}$ is adapted to $(\calF_0,\calF_1)$. We claim that $\mathcal{B}$ is adapted to  $(\calF_0,\calF_2)$. Let us verify each of the conditions in \cref{ad1} for $\mathcal{B}$ and $\calF_0\subseteq \calF_2$. Conditions b) and c) follow directly from Lafont and Ortiz proof, since they do not depend on the families but only on the collection $\mathcal B$. Condition a) also follows directly from Lafont and Ortiz proof, since such a condition only depends on the small family $\calF_0$.
For, condition d), we notice that Lafont and Ortiz verified that every infinite virtually cyclic subgroup of $\Gamma$ is contained in an element of $\mathcal B$, thus we only have to verify that every virtually $\dbZ^2$-subgroup of $\Gamma$ is contained in an element of $\mathcal B$. In fact, if $H$ is a virtually $\Z^2$-subgroup of $\Gamma$, then $H$ does not contain a subgroup isomorphic to a free group in two generators, therefore by Tits alternative for relatively hyperbolic group, see for instance \cite[Remark~3.5]{BW13},  $H$ is subconjugate to the fundamental group of a cusp.

Now we are in a good shape to use  \cref{construction:pusout} and \cref{conditions:exist:pushout} to construct a model for   $E_{\mathcal{F}_{2}}\Gamma$. Let $\mathcal{H}$ be a collection  of representatives of conjugacy classes in $\mathcal{B}$, then the space $X$ defined by the $\Gamma$-push-out given by \cref{construction:pusout}
\begin{equation}
\begin{tikzcd}\label{model:case:hyperbolic}
\displaystyle\bigsqcup_{H\in \mathcal{H}}\Gamma\times_{H}EH \arrow[r, "g"] \arrow[d, "f"]
& E\Gamma  \arrow[d, "h" ] \\
\displaystyle\bigsqcup_{H\in \mathcal{H}}\Gamma\times_{H}E_{\mathcal{F}_{2}}H \arrow[r, "\varphi" ]
& X
\end{tikzcd}
\end{equation}
provides a model for $E_{\mathcal{F}_{2}}\Gamma$. To construct this $\Gamma$-push-out we use the cylinder construction described in \cref{conditions:exist:pushout}.

We claim that $X$ can be chosen to be of dimension  3. Note  that, by \cref{conditions:exist:pushout}, the dimension of $X$ is  \[\max\{\dim(EH)+1, \dim(E\Gamma),\dim(\ef2H)\}.\] We have seen that  $H$ is  isomorphic to  $\Z$ or $\Z^2$, therefore $EH$ has a model of dimension 1 or 2, and in both cases there is a model for $\ef2 H$ of dimension 0.
The hyperbolic space $\dbH^3$ is a 3-dimensional model for $E\Gamma$. This proves our claim. As a consequence $\gd_{\calF_2}(\Gamma)\leq 3$.

Now we are going to show that  $\gd_{\calF_2}(\Gamma)\ge 3$. 
It is well know that  $\gd_{\calF_2}(\Gamma)\ge \cd_{\calF_2}(\Gamma)$, thus it is enough to show that $\cd_{\calF_2}(\Gamma)\ge 3$. To show this last inequality, we will prove that  $H^{3}_{\calF_2}(\Gamma;\underline{\Z})\neq 0$. The  Mayer-Vietoris long exact sequence applied to \eqref{model:case:hyperbolic}  leads to 
\begin{equation}\label{sucesion:mayer:vietoris:1}
\cdots \to H^{3}_{\calF_2}(\Gamma;\underline{\Z}) \to H^{3}(\Gamma;\underline{\Z})\oplus\displaystyle\bigoplus_{H\in \calH}H^{3}_{\calF_2}(H;\underline{\Z})\to \displaystyle\bigoplus_{H\in \calH}H^{3}(H;\underline{\Z})\to \cdots
\end{equation}

Since there is a 0-dimensional model for $\ef2H$ for all $H\in \calH$, thus $H_{\calF_2}^{3}(H;\underline{\Z})=0$, therefore  $\displaystyle\bigoplus_{H\in \calH}H_{\calF_2}^{3}(H;\underline{\Z})=0$. By \cite[Theorem 1.1.7]{sergei:matveev} there exists a 2-dimensional complex $X\subseteq M$ such that $X$ has the same homotopy type of $M$, Thus the universal cover  $\Tilde{X}$ of $X$ is a model  for $E\Gamma$. We conclude that $H^{3}(\Gamma;\underline{\Z})=H^{3}(M;\Z)=H^{3}(X;\Z)=0$.

As a consequence of \eqref{sucesion:mayer:vietoris:1} and the observations in the previous paragraph we get the following exact sequence
\begin{equation*}\label{sucesion:mayer:vietoris:2}
\cdots \to H^{2}(\Gamma;\underline{\Z})\xrightarrow[]{\varphi}  \displaystyle\bigoplus_{\substack{H\in \calH\\ H\cong\Z^2}}H^{2}(H;\underline{\Z})\to H^{3}_{\calF_2}(\Gamma;\underline{\Z})\to 0 
\end{equation*}

Let us describe the map $\varphi$ more explicitly. Recall that this map is induced by the unique $H$ equivariant map $\bigsqcup_{H\in \mathcal{H}, H\cong \dbZ^2} EH \to E\Gamma$ considering $E\Gamma$ as an $H$-CW-complex by restriction, see \cref{conditions:exist:pushout}. To give a concrete description of this inclusion map, we fix as a model for $E\Gamma$ the universal covering $\tilde N$ of the \emph{thick} part $N$ of $M$, that is we chop off all the cusps from $M$. Since every $\Z^2$-subgroup in $\mathcal H$ is the fundamental group of a cusp, thus they correspond to the fundamental groups of the connected components of the boundary of $N$. Therefore the map $\bigsqcup_{H\in \mathcal{H}, H\cong \dbZ^2} \Gamma\times_H EH \to E\Gamma$ can be identified with the inclusion $p^{-1}(\partial N)\to \tilde N$, where $p\colon \tilde{N}\to N$ is the universal covering map. On the other hand we have the natural isomorphisms $H^{2}(\Gamma;\underline{\Z})\cong H^2(E\Gamma/\Gamma;\dbZ)=H^2(N;\dbZ)$ and $H^{2}(H;\underline{\Z})\cong H^{2}(EH/H;\Z)$, see for example \cite[Example~4.3]{ANCMSS21}. In conclusion $\varphi$ can be identified with the map induced in Bredon cohomology by the inclusion $\partial N\hookrightarrow N$.

In order to  show that  $H^{3}_{\calF_2}(\Gamma;\underline{\Z})\neq 0$, it is enough to show that the  map $\varphi$ is not surjective.  Using the long exact sequence of the pair $(N,\partial N)$
\begin{equation*}
\cdots \to  H^{2}(N,\partial N)\to H^{2}(N)\xrightarrow[]{\varphi}H^{2}(\partial N)\to H^{3}(N,\partial N)\to 0
\end{equation*}
we can see that $\varphi$ is not surjective if and only if $H^{3}(N,\partial N)\neq 0$. By Poincaré duality  $H^{3}(N,\partial N)\cong H_{0}(N)=\Z$. Therefore $\varphi$ is not surjective. This finishes the proof.
\end{proof}

\section{The Seifert fiber case}\label{section:seifert}

In this section we compute the $\calF_2$-dimension of Seifert fiber manifolds with and without boundary. These computations are carried out in  \cref{thm:seifer:base:spherical}, \cref{thm:seifer:base:hyperbolic}, \cref{thm:seifer:base:flat}, and \cref{thm:seifer:base:flat:bondary}

\subsection{ Seifert fiber manifolds with bad or spherical base orbifold}

\begin{theorem} \cite[Proposition 5.1]{Luis:Lafon:Joeken}\label{thm:seifer:base:spherical}
Let $M$ be a  closed Seifert fiber 3-manifold with base orbifold $B$ and fundamental group $\Gamma$. Assume that $B$ is either a bad orbifold, or  a good orbifold modelled on $\mbs^{2}$. Then $\Gamma$ is virtually cyclic, in particular  $\gd_{\mathcal{F}_{k}}(\Gamma)=0$ for all $k\ge 1$. 
\end{theorem}

\subsection{ Seifert fiber manifolds with hyperbolic base orbifold}
\begin{theorem}\label{thm:seifer:base:hyperbolic}
Let $M$ be  a Seifert fiber 3-manifold (possibly with non-empty boundary) with base orbifold $B$ and fundamental  group  $\Gamma$. Asumme that  $B$ is modelled on  $\maH^2$, then $\gd_{\calF_2}(\Gamma)=2.$
\end{theorem}

\begin{proof}
First, we are going to show that $\gd_{\calF_2}(\Gamma)\le 2$, for this we will construct a model for  $E_{\calF_2}\Gamma$ of dimension 2 using \cref{construction:pushout:pullbak}. By \cref{sequence:exact:groups:fundamental:seifert} we obtain an exact sequence of groups \begin{equation}\label{sequence:groups:seifert}1\to K \to \Gamma \xrightarrow{\varphi} \Gamma_0\to 1,\end{equation}
where $\Gamma_0=\pi_1^{orb}(B)$ is  the  orbifold fundamental group of  $B$ and $K$ is a cyclic infinite subgroup of  $\Gamma$ generated by a regular fiber.
Let $\calF_n^{\prime}=\{H<\Gamma_0: H \ \text{is  virtually} \Z^r \ \text{for some}\ r=0,1,\dots,n\}$. Note that, since $\Gamma_0$ is a  Fuchsian group,  it  cannot have a subgroup isomorphic to $\Z^2$, therefore $\calF_2^{\prime}=\calF_1^{\prime}$. By \cite[Theorem 2.6]{lafont:ortiz} the collection $\calA=\{H<\Gamma_0: H \ \text{is virtually} \ \Z \ \text{ and maximal in} \ \calF_1^{\prime}- \calF_0^{\prime}\}$ is adapted to the pair  $(\calF_0^{\prime},\calF_1
^{\prime})$ .

Consider the pulled-back families  $\tilde{\calF_0^{\prime}}$ and $\tilde{\calF_1^{\prime}}$ which by definition are generated by  $\{\varphi^{-1}(L): L \in \calF_0^{\prime}\}$ and $\{\varphi^{-1}(L): L \in \calF_1^{\prime}\}$ respectively. We claim that $\tilde{\calF_0^{\prime}}\subset\calF_2\subset \tilde{\calF_1^{\prime}}$.  The first inclusion follows from the fact that $\varphi^{-1}(L)$, with $L\in \calF_0'$, is a virtually cyclic subgroup of $\Gamma$. While the second follows from he following argument. For $L\in \calF_2$, we have three options:  $L$ is finite, $L$ is  virtually $\Z$ or $L$ is  virtually $\Z
^2$. If $L$ is virtually $\Z^2$, then there is a  subgroup $L^{\prime}<L$ of finite index isomorphic to $\Z^{2}$, note that  $\varphi(L^{\prime})$ is a cyclic subgroup of $\Gamma_0$ that  has finite index in $\varphi(L)$, thus  $\varphi(L)\in \calF_1^{\prime}$. We conclude that  $L\subseteq\varphi^{-1}(\varphi(L))\in \tilde{\calF_1^{\prime}}$. The other cases follow using a completely analogous argument. Let $\calH$ be  a complete set of representatives of conjugacy classes in $\calA^{*}=\{\varphi
^{-1}(H):H\in \calA\}$. Then by  \cref{construction:pushout:pullbak} we have the following $\Gamma$-push-out
gives a model  $X$ for   $E_{\calF_2}\Gamma$

\[\xymatrix{\coprod_{\tilde{H}\in\calH}\Gamma\times_{\tilde{H}}E_{\calF_0}H \ar[d] \ar[r]  & E_{\calF_0}\Gamma_0\ar[d]\\
\coprod_{\tilde{H}\in\calH}\Gamma\times_{\tilde{H}}E_{\calF_2}\tilde{H}\ar[r] & X
}
\]
where $\tilde H$ stands for $\varphi^{-1}(H)$. In the above push-out, if needed, we can replace the upper horizontal arrow by its mapping cylinder to satisfy the conditions at the end of the statement of \cref{construction:pushout:pullbak}, see \cref{conditions:exist:pushout}.

We claim that $X$ can be choosen to be of dimension  2. Note  that the dimension of $X$ is  \[\max\{\dim(E_{\calF_0}H)+1, \dim(E_{\calF_0}\Gamma_0),\dim(\ef2\tilde{H})\}.\] 
Since $H\in \calA$ we have that  $H$ is virtually  $\Z$, then a model for  $E_{\calF_0}H$ is $\mathbb{R}$ for all $\tilde{H}\in \mathcal{H}$. As a consequence of \cite[corollary 2.8]{Bridson:Haefliger} a model for  $E_{\calF_0}\Gamma_0$ is $\maH^2$. Finally,  we show that  $\tilde{H}$ is virtually $\Z
^2$ for all $\tilde{H}\in \mathcal{H}$, and in consequence the one point space is a model for  $E_{\calF_2}\tilde{H}$ for all $\tilde{H}\in \mathcal{H}$. Let $\tilde{H}\in \calH$, by definition of  $\calH$, $\tilde{H}\in \calA^{*}$, then $\tilde{H}=\varphi^{-1}(S)$ for some $S\in \calA$. The group $S$ is virtually $\Z$, i.e.  there is a subgroup $T<S$ such that  $T$ is a finite index subgroup of $S$ isomorphic to $\Z$. Thus by  \eqref{sequence:groups:seifert} we have the following short exact sequence 
\begin{equation*}\label{sequence:groups}
    1\to K\to \varphi^{-1}(T)\to T\to 1,
    \end{equation*}
then $\varphi^{-1}(T)$ is isomorphic to  $\Z\rtimes \Z$ which is virtually $\Z^2$. It follows that $\tilde{H}$ is virtually $\Z^2$. Therefore $X$ can be constructed to be a 2-dimensional $\Gamma$-CW-complex.

The group $\Gamma$ is finitely generated as it is the fundamental group of a compact manifold. On the other hand $\Gamma\notin \calF_2$ since it has as a quotient the group $\Gamma_0$ \eqref{sequence:groups:seifert}, which is either  virtually non-cyclic free or virtually a surface group. Therefore,  by \cref{lemma:low:dimensions}, $\gd_{\calF_2}(\Gamma)\ge 2$.
\end{proof}

\subsection{ Seifert fiber manifolds with flat base orbifold}
\begin{theorem}\label{thm:seifer:base:flat} 
Let $M$ be a Seifert fiber closed 3-manifold  (without boundary) with base orbifold $B$ and  fundamental group  $\Gamma$. Suppose that  $B$ is modelled on  $\E^2$. Then $M$ is  modelled  on  $\E^3$ or is modelled on  $\nil$. Moreover, 
\begin{enumerate}[a)]
    \item If $M$ is modelled on $\E^3$, then $\gd_{\calF_2}(\Gamma)=5$.
    \item If $M$ is modelled on $\nil$, then $\gd_{\calF_2}(\Gamma)=3$.
\end{enumerate}
\end{theorem}

\begin{theorem}\label{thm:seifer:base:flat:bondary}
Let $M$ be a compact Seifert fiber 3-manifold, with base orbifold  $B$ and fundamental group $\Gamma$. Suppose that  $B$ is modelled on $\E^2$ and $M$ has non-empty boundary. Then  $M$ is  diffeomorphic to $T^2\times I$ or the twisted $I$-bundle over Klein bottle. In particular, $\gd_{\calF_k}(\Gamma)=0$ for all $k\ge 2$.
\end{theorem}

Before proving the above theorems, let us set some notation. For $A\in\gldosz$, we say $A$ is

\begin{enumerate}[a)]
    \item \emph{Elliptic}, if it has finite order.
    
    \item \emph{Parabolic}, if it is conjugate to a matrix of the form

\begin{equation}\label{matrix:canonical:type2}
    \begin{pmatrix}
    1&s\\
    0&1
    \end{pmatrix}
    \end{equation}
    for some $s\neq 0$.
    
    \item \emph{Hyperbolic}, if it is not elliptic nor parabolic.
\end{enumerate}

It is well-known that every matrix in $A\in\gldosz$ is either elliptic, parabolic or hyperbolic. Moreover, $A$ is elliptic (resp. parabolic, hyperbolic) if and only if  $A^r$ is elliptic (resp. parabolic, hyperbolic) for all $r\geq 1$.

\begin{theorem}\label{thm:case:plane}
Let $M$ be a  $3$-manifold  with fundamental group  isomorphic  to $\Gamma=\Z^2\rtimes_{\varphi} \Z$ where $\varphi \colon \Z \to \Aut(\Z^2)=\gldosz$ is a homomorphism. Then the following statements hold.
\begin{enumerate}[a)]
\item If $\varphi(1)$ is elliptic, then $\gd_{\calF_2}(\Gamma)=5$. 
\item If $\varphi(1)$  is parabolic, then  $\gd_{\calF_2}(\Gamma)=3$.
\item If $\varphi(1)$ is  hyperbolic, then  $\gd_{\calF_2}(\Gamma)=2$.
\end{enumerate}
\end{theorem}

For the proof of \cref{thm:case:plane} we need the following three lemmas.

\begin{lemma}\label{dg:group:virtually:Z3:F2}
Let $\Gamma$ be a  virtually  $\Z^3$ group, then $\gd_{\calF_2}(\Gamma)=5$.  \end{lemma}
\begin{proof}
 In  \cite[Proposition A]{onorio} they show that  $\gd_{\calF_2}(\Z^3)= 5$, in consequence   $\gd_{\calF_2}(\Gamma)\ge 5$. On the other hand, in  \cite[Proposition 1.3.]{prytua2018bredon} and \cite[Theorem~6.14]{Mo19} it is proven that $\gd_{\calF_2}(\Gamma)\le 5$. 
Thus $\gd_{\calF_2}(\Gamma)=5$.
\end{proof}

\begin{lemma}\label{lemma:case:hyperbolic}
 Let $\Gamma=\Z^2\rtimes_{\varphi} \Z$ with $\varphi(1)=A$ a parabolic element in $GL_2(\Z)$. Then
 \begin{enumerate}[a)]
     
 \item The matrix $A$ fixes a infinite cyclic subgroup of $\Z^2$. Moreover  the infinite cyclic maximal subgroup $N$ fixed by  $A$ is normal in 
 $\Gamma$ and $\Gamma/N$ is isomorphic to  $\Z^2$ or to $\Z\rtimes \Z$.
 \item Consider the homomorphism $\pi \colon \Gamma \to \Gamma/N$ and let $\calF_1^{\prime}$ be the family of virtyally cyclic subgroups of $\Gamma/N$. Then the pulled-back family  $\tilde{\calF_1^{\prime}}$ of $\calF_1^{\prime}$ is equal to the family $\calF_2$ of $\Gamma$.
 \end{enumerate}
 \end{lemma}
 
 \begin{proof}First we prove part  $a)$.
Since $A$ is parabolic, without loss of generality we can assume that $A$ is a matrix of the form \eqref{matrix:canonical:type2}. Therefore, $A$ fixes a maximal infinite cyclic subgroup $N$ of $\Z^2$. It is easy to see that  $N$ lies in the center of $\Gamma$, and it follows that $N$ is a normal subgroup of $\Gamma$. Note that  $\Gamma/N=(\Z^2/N)\rtimes \Z$  is isomorphic to $\Z^2$ or to $\Z\rtimes \Z$.

Now we proof part $b)$. First we show the inclusion   $\tilde{\calF_1^{\prime}}\subseteq \calF_2$. Note that both $\Gamma$ and $\Gamma/N$ are torsion-free groups, and in consequence every virtually cyclic subgroup of them is either trivial or infinite cyclic. By definition of $\tilde{\calF_1^{\prime}}$ it is enough to prove that $\pi^{-1}(L)\in \calF_2$ for every infinite cyclic subgroup $L$ of $\Gamma/N$. Let $L$ be an infinite cyclic subgroup of $\Gamma/N$,  hence $\pi^{-1}(L)$ fits in the following short exact sequence
 \begin{equation*}
    1\to \Z\to \pi^{-1}(L)\to \Z \to 1 \end{equation*}
and we conclude that $\pi^{-1}(L)$ is isomorphic to $\Z^2$ or to $\Z\rtimes \Z$. In either case   $\pi^{-1}(L)$  is virtually $\Z^2$.

Now we show the inclusion $\calF_2\subseteq\tilde{\calF_1'}$. Let $L\in \calF_2$, we have three cases:  $L$ is trivial, $L$ is infinite cyclic, or $L$ is virtually $\Z^2$. If $L$ trivial, there is nothing to prove. If  $L$ is isomorphic to  $\Z$, then $\pi(L)$ is trivial or infinite cyclic, and therefore $\pi(L)\in \mathcal{F}_1^{\prime}$. Thus  $L\subseteq \pi^{-1}(\pi(L))\in\tilde{\calF_1^{\prime}}$. Finally,  suppose that $L$ is virtually  $\Z^2$, i.e there is a finite index subgroup  $T<L$ such that  $T$ is isomorphic $\Z^2$. We claim that  $T \cap N $ is non-trivial. Suppose $T \cap N=1 $, since $N$ is a subgroup of the center of $\Gamma$, a generator  $h$ of  $N$ commutes with generators $\alpha$ and  $\beta$ of  $T$,  thus the subgroup generated by $h,\alpha, \beta$ is isomorphic to $\Z^3$ and by \cite[Theorem 5.13 (iii)]{luck:Weiermann} this subgroup has $\mathcal{F}_1$-dimension equal to 4. In consequence  $\gd_{\mathcal{F}_1}(\Gamma)\ge 4$, but this contradicts the fact that $\gd_{\mathcal{F}_1}(\Gamma)=3$ (see \cite[Proposition 5.4]{Luis:Lafon:Joeken}).  From our claim since $\Gamma/N$ is torsion-free, we can see that  $\pi(T)$ is infinite cyclic. Since  $\pi(T)$ is of finite index in   $\pi(L)$, it follows that  $\pi(L)$ is virtually $\Z$, then $\pi(L)\in \calF_1'$. Therefore   $L<\pi
^{-1}(\pi(L))\in \tilde{\calF_1'}$. 
 \end{proof}
 
\begin{lemma}\label{subgroups:semi:direct}Let $\Gamma=\Z^2\rtimes_{\varphi} \Z$ with $\varphi(1)=A$ a hyperbolic element in $GL_2(\Z)$. Let $H$ be the subgroup $\Z^2\rtimes\{0\}$  of $\Gamma$. Then
\begin{enumerate}[a)]
    \item Every subgroup of $\Gamma$ isomorphic to $\Z^2$ is a subgroup of  $H$. In particular, the family  $\calF_2=\calF_1\cup \SUB(H)$, where  $\SUB(H)$ is the family of all subgroups of  $H$.
  \item Let $C$ be an infinite cyclic subgroup of  $\Gamma$, then
  
  \[
  N_\Gamma C\cong 
  \begin{cases}
   \dbZ & \text{ if } C\not\le H\\
   \dbZ^2 & \text{if } C\le H
  \end{cases}
  \]
\end{enumerate}
 \end{lemma}
 \begin{proof}
First we proof the part  $a)$ by  contradiction. Suppose that there is  $L$  a subgroup of $\Gamma$ isomorphic to  $\Z^2$ that is not subgroup of $H$.  From  the short exact sequence   
\begin{equation}\label{sequence:exact:case3}
    1\to  H\to \Gamma\xrightarrow{\psi} \Z \to 1 \end{equation}
we obtain the following short exact sequence   
 \begin{equation*}\label{sequence:exact:case3:r}
    1\to L\cap H\to L\to \psi(L) \to 1 \end{equation*}
Since  $L$ is not a subgroup of $H=\ker \psi$, we have that  $\psi(L)$ is non-trivial, then  $\psi(L)=\langle r\Z\rangle$ with  $r\neq 0$. Since  $L$ is isomorphic to  $\Z^2$ and $\psi(L)=\langle r\Z\rangle$,  it follows that  $L\cap H$ is isomorphic to  $\Z$, in consequence   $L=(L\cap H)\times \langle r\Z\rangle$. Thus $A^r$ fixes an infinite cyclic subgroup of $H\cong \Z^2$, and therefore  $A^r$ is parabolic, which implies that $A$ is parabolic
. This is a contradiction since we were assuming that $A$ is hyperbolic.

Now we going to show that  $\calF_2=\calF_1\cup \SUB(H)$. It is clear that $\calF_1\cup \SUB(H) \subseteq \calF_2$. It only remains to prove that   $\calF_2\subseteq \calF_1\cup \SUB(H)$. Let  $L\in \calF_2$ then  $L$ is either virtually cyclic or $L$  is virtually  $\Z^2$. If  $L$ is virtually cyclic, by definition  $L\in \calF_1\subseteq  \calF_1\cup \SUB(H)$. Now suppose that $L$ is virtually $\Z^2$. From \eqref{sequence:exact:case3} we get the following  short exact sequence  
 \begin{equation}\label{sequence:exact:case:plane}
    1\to L\cap H\to L\to \psi(L) \to 1. \end{equation}
Since $L$ is virtually  $\Z^2$ and, by part a),  $H$ contains all the subgroups of $\Gamma$ isomorphic to $\Z^2$ we have that $L\cap \Z^2$ is isomorphic to  $\Z^2$. Then \eqref{sequence:exact:case:plane} is equivalent to
\begin{equation*}\label{10}
    1\to \Z^2\to L\to \psi(L) \to 1. \end{equation*}
Once more, since $L$ is virtually $\Z^2$ we conclude that $\psi(L)$ must be trivial. It follows that $L\in \SUB(H)\subseteq \calF_1\cup \SUB(H)$.

Now we prove  part  $b)$. First suppose  $C\not\le H$, then the elements  of $C$ are the form $((0,0),l)$. Let $((x,y),w)\in N_{\Gamma}C$, then
     \[
     \begin{split}
     ((x,y),w) ((0,0),l)((x,y),w)^{-1}
    =&((x,y)+A^{w}(-x,-y),l)
     \end{split}
     \]
It follows that 
    $A^{w}(-x,-y)= (-x,-y)$. By hypothesis $A$ is hyperbolic, thus $A^{w}$ is hyperbolic, therefore $(-x,-y) =(0,0)$.
We conclude that  $N_{\Gamma}C=C\cong\Z$.  
    
Now suppose that   $C\le H$. From \eqref{sequence:exact:case3} we have the following short exact sequence 
\begin{equation*}
    1\to N_{\Gamma}C\cap H\to N_{\Gamma}C\to \psi(N_{\Gamma}C) \to 1 \end{equation*}
by hypothesis $C \le H$, then $H<N_{\Gamma}C$, thus the sequence above is equivalent to \begin{equation}\label{11}
    1\to  H\to N_{\Gamma}C\to \psi(N_{\Gamma}C) \to 1 \end{equation}
We are going to show that  $ \psi(N_{\Gamma}C)=0$, for this is enough to see that $N_{\Gamma}C$ does not contain elements of the form $((a,b),l)$ with $l\neq 0$. Suppose that $C$ is generated by $((x,y),0)$, then 
\[
     \begin{split}
     ((a,b),l)((x,y),0)((a,b),l)^{-1}=&(A^{l}(x,y),0)
     \end{split}
     \]
It follows that $A^{l}(x,y)=\pm (x,y)$. But we are assuming that $A$ is hyperbolic which implies that $A^{l}$ is hyperbolic. Therefore $\psi(N_{\Gamma}C)=0$. We conclude of \eqref{11} that $N_\Gamma C\cong \Z^2$. 
\end{proof}

\begin{proof}[Proof of \cref{thm:case:plane}]
$a)$ We claim that $\Gamma$ is virtually $\Z^3$. By hypothesis $A$ is elliptic, then $A$ has finite order. Let  $n$ be  the order of $A$, then the subgroup $\{0\}\rtimes \langle n\Z \rangle$  acts trivially on the factor $\Z^2$. Therefore $\Z^2\rtimes \langle n\Z \rangle= \Z^2\times \langle n\Z \rangle$ is a finite index subgroup of  $\Z^2\rtimes_{\varphi} \Z$ isomorphic to   $\Z^3$. Now the claim follows from \cref{dg:group:virtually:Z3:F2}.

$b)$ First  we construct a model for $E_{\calF_2}\Gamma$ of dimension  3, and as a consequence we get $\gd_{\calF_2}(\Gamma) \le 3$.  By \cref{lemma:case:hyperbolic} part $a)$ the matrix  $A$ fixes a maximal  infinite cyclic subgroup $N$   and  $\Gamma/N=(\Z^2/N)\rtimes \Z$  is isomorphic to either $\Z^2$ or  $\Z\rtimes \Z$. In both cases $\Gamma/N$ is 2-crystallographic group, then by \cite{Conolly:Frank:Hartglas} there is a   3-dimensional model $Y$ for $E_{\calF_1
^{\prime}}(\Gamma/N)$ where $\calF_1^{\prime}$ is the family of  virtually cyclic subgroups of  $\Gamma/N$. By  \cref{lemma:case:hyperbolic} part $b)$ we have  $\tilde{\calF_1^{\prime}}=\calF_2$, then   $Y$ is also a model for  $E_{\calF_2}\Gamma$ with the $\Gamma$-action induced by  $\pi\colon \Gamma\to\Gamma/N$. 

Now we show  $\gd_{\calF_2}(\Gamma)\ge 3$. It is enough to prove that  $H_{\calF_2}^{3}(\Gamma, \underline{\Z})\neq 0$. Let $Y$ be the model for  $E_{\calF_2}\Gamma$ that we  mentioned in the previous paragraph, then 

\[\begin{split}
    H_{\calF_2}^{3}(\Gamma; \underline{\Z})&=  H^{3}(Y/\Gamma; \Z)\\
    &=H^{3}(Y/(\Gamma/N); \Z).
    \end{split}
\]
The latter homology group was shown to be nonzero in \cite[p.8 proof Theorem 1.1]{Conolly:Frank:Hartglas}.

$c)$ In order to prove  $\gd_{\calF_2}(\Gamma)\le 2$, we construct a model for $E_{\calF_2}\Gamma$ of dimension 2. Let $H$ denote $\Z^2\rtimes \{0\}$. By \cref{subgroups:semi:direct} part $a)$ the family  $\calF_2$ of $\Gamma$ is the union  $\calF_2=\calF_1\cup \SUB(H)$ where  $\SUB(H)$ is the family of all subgroups of $H$. By \cite[Lemma 4.4]{DQR11} the space $X$ given by the following \emph{homotopy} $\Gamma$-push-out  is a model for $E_{\calF_2}\Gamma$
 \begin{equation}\label{350}
 \xymatrix{E_{\calF_1(H)}\Gamma\ar[d]^g \ar[r]^f  & E_{\calF_1}\Gamma\ar[d]\\
E_{\SUB(H)}\Gamma \ar[r] & X
}
\end{equation}
where $\calF_1(H)=\SUB(H)\cap \calF_1$. Here the maps $g$ and $f$ are the unique (up to $\Gamma$-homotopy) $\Gamma$-maps given by the inclusion of families $\calF_1(H)\subseteq \calF_1$ and $\calF_1(H)\subseteq \SUB(H)$. We claim that, with suitable choices, $X$ is 2-dimensional.Let us explain first the idea of the construction. First note that $H$ is a normal subgroup of $\Gamma$,  $E(\Gamma/N)=E\Z=\R$ is a model for $E_{\SUB(H)}\Gamma$, see for example the proof of Corollary 2.10 in \cite{luck:Weiermann}. Next, we will show that there is a 3-dimensional model $Y$ for $E_{\calF_1}\Gamma$ such that 
\begin{enumerate}
    \item $Y$ is the union of two $\Gamma$-subcomplexes $Y_1$ and $Y_2$.
    
    \item Every 3-cell in $Y$ belongs to $Y_2$.
    
    \item $Y_2$ is model for $E_{\calF_1(H)}\Gamma$.
    \item The map $f$ is the inclusion $Y_2\to Y$.
\end{enumerate}
Since $f$ is an inclusion, then the homotopy $\Gamma$-push-out \eqref{350} can be replaced by an honest $\Gamma$-pushout, see for example \cite[Theorem~1.1]{Wa80}. That is, we take $X=Y\cup_g \mathbb R=Y\sqcup \mathbb R/\sim$, where we identify $x\sim g(x)$ for all $x\in Y_2$. By (1)-(4), we can see that every 3-cell of $Y$ is being collapsed to a 1-dimensional space via $g$, hence we get that $X$ is a complex of dimension less than or equal to 2. From the explicit construction we actually conclude $X$ is of dimension 2.

Let us construct $Y$, $Y_1$ and $Y_2$.  Since $\Gamma$ is a torsion-free poly-$\Z$ group, by  \cite[Lemma 5.15 and Theorem 2.3]{luck:Weiermann} we get a model $Y$ for $E_{\calF_1} \Gamma$ via the following $\Gamma$-push-out 
\begin{equation}\label{diagram}
 \xymatrix{\displaystyle\bigsqcup_{C\in \I}\Gamma\times_{N_\Gamma C} EN_{\Gamma}C\ar[d] \ar[r]  & E\Gamma\ar[d]\\
\displaystyle\bigsqcup_{C\in \I}\Gamma\times_{N_\Gamma C} E W_{\Gamma}C \ar[r] & Y
}
\end{equation}
where  $\I$ is a set of   representatives of commensuration classes in $\calF_1-\calF_0$. To construct the aforementioned $\Gamma$-pushout we use a construction analogue to the one described in \cref{conditions:exist:pushout}, that is we replace $E\Gamma$ with the mapping cylinder of the top horizontal arrow.

By \cref{subgroups:semi:direct}  part $b
)$ the set $\I$ is the disjoint union of $\I_1=\{C\in \I | N_{\Gamma}C\cong\Z \}=\{C\in \I | C\not\le \Z^2\rtimes \{0\} \}$ and $\I_2=\{C\in \I | N_{\Gamma}C\cong\Z^2\}=\{C\in \I | C\le \Z^2\rtimes \{0\} \}$. 
Therefore, the  $\Gamma$-push-out \eqref{diagram}  can be written as 
\begin{equation}\label{pushout:space:classifying:F1}
 \xymatrix{\left(\displaystyle\bigsqcup_{C\in \I_1}\Gamma\times_{N_\Gamma C} EN_{\Gamma}C\right)\displaystyle\bigsqcup\left(\displaystyle\bigsqcup_{C\in \I_2}\Gamma\times_{N_\Gamma C} EN_{\Gamma}C\right)\ar[d] \ar[r]  & E\Gamma\ar[d]\\
\left(\displaystyle\bigsqcup_{C\in \I_1}\Gamma\times_{N_\Gamma C} EW_{\Gamma}C\right) \displaystyle\bigsqcup\left(\displaystyle\bigsqcup_{C\in \I_2}\Gamma\times_{N_\Gamma C} EW_{\Gamma}C\right)  \ar[r] & Y
}
\end{equation}

Define $Y_1$ and $Y_2$ via the following $\Gamma$-push-outs

\begin{equation*}
 \xymatrix{\displaystyle\bigsqcup_{C\in I_1}\Gamma\times_{N_\Gamma C} EN_{\Gamma}C\ar[d] \ar[r]  & E\Gamma\ar[d]\\
\displaystyle\bigsqcup_{C\in I_1}\Gamma\times_{N_\Gamma C} EW_{\Gamma}C  \ar[r] & Y_1 
}\quad\quad\quad\quad\quad
\xymatrix{\displaystyle\bigsqcup_{C\in I_2}\Gamma\times_{N_\Gamma C} EN_{\Gamma}C\ar[d] \ar[r]  & E\Gamma\ar[d]\\
\displaystyle\bigsqcup_{C\in I_2}\Gamma\times_{N_\Gamma C} EW_{\Gamma}C  \ar[r] & Y_2 
}
\end{equation*}
We used the cylinder construction described in  \cref{conditions:exist:pushout} to construct all the pushouts above.
Clearly $Y=Y_1\cup Y_2$, and $Y_1\cap Y_2=E\Gamma$. Since $\calF_1(H)$ is the family of all virtually cyclic subgroups of $H$ and $I_2$ is set of representativs of commensuration classes of infinite cyclic subgroups in $H$, we can  use   \cite[Lemma 5.15 and Theorem 2.3]{luck:Weiermann} to prove that $Y_2$ is a model for $E_{\calF_1(H)}\Gamma$.

We note the following:
\begin{enumerate}[i)]
    \item If $C\in \I_1$ then a model for  $EN_{\Gamma}C\cong E\dbZ$ is $\mathbb{R}$. Thus the cylinder associated to $EN_{\Gamma}C\cong E\dbZ$  contributes to $E_{\calF_1}\Gamma$ with a subspace of dimension 2. 
    \item If $C\in \I_2$ then a model for  $EN_{\Gamma}C\cong E\dbZ^2$ is $\mathbb{R}^2$. Thus the cylinder associated to $EN_{\Gamma}C$ contributes to $E_{\calF_1}\Gamma$ with a subspace of dimension 3.
    \item By \cite[Theorem 5.13]{luck:Weiermann} $E\Gamma$ has a model of dimension 3.
\end{enumerate}  

As a consequence of the above observations, we conclude that every 3-cell of $Y$ belongs to $Y_2$. This concludes the construction of a 2-dimensional model for $E_{\calF_2}\Gamma$. 

Clearly $\Gamma$ is finitely generated and $\Gamma\notin\calF_2$, therefore by \cref{lemma:low:dimensions}, $\gd_{\calF_2}(\Gamma)\geq 2$.
\end{proof}

\begin{proof}[Proof of \cref{thm:seifer:base:flat}]
The first conclusion follows from  \cite[Theorem 5.3]{Sc83}. For part $a)$, we have that  $\Gamma$ is 3-crystallographic, hence $\Gamma$ is virtually $\Z^3$. Hence part $a)$ follows from \cref{dg:group:virtually:Z3:F2}.

 For part  $b)$, we have that  $\Gamma$ is  torsion-free as it is the fundamental group of a spherical manifold, then by \cite[Theorem N]{Charles:Thomas} we obtain a short exact sequence  
 $$1\to  \Z^2\to \Gamma \to \Z \to 1 $$
where the morphism  $\varphi\colon \Z \to \Aut(\Z^2)$ is such that $\varphi(1)=A$ is parabolic. Hence by the  \cref{thm:case:plane} part  $b)$  we have  $\gd_{\calF_2}(\Gamma)=3$.
\end{proof}

\begin{proof}[Proof of \cref{thm:seifer:base:flat:bondary}]
By  \cite[Theorem 1.2.2]{John:W:Morgan} $M$ is modelled on  $\E^3$, $\nil$ or  $M$ is diffeomorphic to  $T^2\times I$ or the twisted $I$-bundle over the  Klein bottle. The 3-manifold $M$ cannot be modelled on  $\E^3$ and $\nil$ as any such manifold must have empty boundary, see \cite[p. 60, paragraph 2 and last paragraph]{John:W:Morgan}. We conclude that $M$ is diffeomorphic to $T^2\times I$ or to the  twisted $I$-bundle over the  Klein bottle. The fundamental groups of  $T^2\times I $ and the   twisted $I$-bundle over the  Klein bottle are isomorphic to  $\Z^2$ and  $\Z\rtimes\Z$ respectively. In either case $\Gamma$ is virtually $\Z^2$ and therefore  $\gd_{\calF_2}(\Gamma)=0$.
\end{proof}

\section{The Sol case}\label{section:sol}

In this section we compute the $\calF_2$-dimension of manifolds modelled on $\sol$. This computation is relevant in the statement and proof of \cref{geometric:dimension:prime}.

\begin{proposition}\label{dg:manifold:prime:sol}
Let $M$ be a connected closed  3-manifold  modelled on  $\sol$ with fundamental group $\Gamma$. Then $\gd_{\calF_2}(\Gamma)=2$.
\end{proposition}

In order to prove \cref{dg:manifold:prime:sol}, we need the following lemma. Let $\calK$ be the fundamental group of the Klein bottle.

\begin{lemma}\label{lemma:case:sol} Let $K_1$ and $K_2$ be copies of $\calK$, and let $A$ be the index two $\Z^2$-subgroup of $\calK$. Assume $\varphi\colon A \to A$ is a hyperbolic isomorphism. Consider $\Gamma=K_1*_{A}K_2$ the amalgamated product associated to $\varphi$. Then

\begin{enumerate}[a)]
    \item $\calF_2=\calF_1\cup \calF(K_1, K_2)$, where $\calF(K_1, K_2)$ is the smallest family of $\Gamma$ containing $K_1$ and $K_2$.
    \item Let $C$ be an infinite cyclic subgroup of $\Gamma$ then 
     \[
  N_\Gamma C\cong 
  \begin{cases}
   \text{is virtually } \Z^2 & \text{ if } |C\cap A|=\infty \\
   \text{is virtually } \Z & \text{} otherwise.
  \end{cases}
  \]
    
\end{enumerate}
\end{lemma}
\begin{proof}
a) We show the inclusion $\calF_2\subseteq \calF_1\cup \calF(K_1,K_2)$. Let $S\in\calF_2 $ then $S$ is virtually cyclic or virtually $\Z^2$. If $S$ is virtually cyclic by definition $S\in \calF_1\subseteq\calF_1\cup \calF(K_1,K_2)$. Now, let $S$ be  virtually $\Z^2$. We have the following short exact sequence \begin{equation*}\label{short:exact:sequence:lemma:1}
1\to  A \to \Gamma \xrightarrow{\psi} D_{\infty} \to 1
\end{equation*}
Let $L$ be the subgroup isomorphic to $\Z$ of $D_{\infty}$ of index 2. Then the subgroup $\Gamma^{\prime}=\psi^{-1}(L)\cong A\rtimes_{\varphi}L\le\Gamma$ is also  of index 2. It follows that $\Gamma^{\prime}\cap S$ is finite index in $S$. Since  $S\le \Gamma$ is virtually $\Z^2$, then  $\Gamma^{\prime}\cap S\le \Gamma^{\prime}$ is also virtually $\Z^2$. By \cref{subgroups:semi:direct} $, \Gamma^{\prime}\cap S\le A$. Therefore in the next short exact sequence  
\begin{equation*}
1\to  A\cap S \to S \to \psi(S) \to 1
\end{equation*}
we have that $\psi(S)$ is finite. If we think of $D_\infty$ as the free product of two cyclic groups of order two, we can easily see that every finite subgroup of $D_\infty$ is either trivial or a conjugate of one of the free factors. It follows that $S$ is subconjugate to $K_1$ or $K_2$. 
The other inclusion $\calF_1\cup \calF(K_1,K_2)\subseteq \calF_2$ is clear.

b) First assume $|C\cap A|=\infty$. Recall that if $C$ and $C'$ are commensurable cyclic subgroups of $\Gamma$, then $N_\Gamma C=N_\Gamma C'$, see \cite[Lemma~5.15]{luck:Weiermann}, thus we can assume that $C\le A$. Then by \cref{subgroups:semi:direct}, $N_{\Gamma}C\cap \Gamma^{\prime}$ is isomorphic to $\Z^2$. Now $\Gamma^{\prime}$ is of index 2 in $\Gamma$, then $[N_{\Gamma}C:N_{\Gamma}C\cap\Gamma^{\prime}]\leq 2$. It follows that $N_{\Gamma}C$ is virtually to $\Z^2$. 
Now suppose that $|C\cap A|<\infty$, thus $C\cap A$ is trivial because $A$ is torsion-free. By \cref{subgroups:semi:direct}, we have $N_{\Gamma}C\cap \Gamma^{\prime}$ is isomorphic to $\Z$. Now $\Gamma^{\prime}$ is of index 2 in $\Gamma$, then $[N_{\Gamma}C:N_{\Gamma}C\cap\Gamma^{\prime}]\leq2$. It follows that $N_{\Gamma}C$ is virtually to $\Z$.

\end{proof}

\begin{proof}[Proof of \cref{dg:manifold:prime:sol}]

In \cite[Theorem 1.8.2, p.17]{matthias:stefan:henry} they show that if $M$ is modelled on $\sol$, then either $M$ is the mapping torus of $(T^2,A)$ with $A$ Anosov or $M$ is a double of  $K\tilde{\times} I$, the twisted $I$-bundle over the Klein bottle. In the first case  we have  that $\Gamma=\Z^2\rtimes_{\psi}\Z$ with $\psi(1)=A$ hyperbolic. By \cref{thm:case:plane} part $c)$  $\gd_{\calF_2}(\Gamma)=2$.

Suppose now that $M$ is a double of $K\tilde{\times} I$. In this case, by Seifert-Van Kampen theorem,   $\Gamma=K_1\ast_{\Z^2}K_2$ where $K_1$ and $K_2$ are copies of the fundamental group of the Klein bottle  $\calK$, and $\dbZ^2$ is embedded in $K_i$ as an index 2 subgroup. Let $H$ be the index two $\dbZ^2$-subgroup of $\calK$. In the proof of \cref{lemma:case:sol} we proved that $\Gamma$ contains an index 2 subgroup isomorphic to $\Z^2\rtimes_\varphi \Z$ with $\varphi\colon H\to H$ a hyperbolic isomorphism. It follows, from our previous case  $\gd_{\calF_2}(\Gamma)\ge \gd_{\calF_2}(\Z^2\rtimes_{\varphi}\Z)=2$. It remains to be proven that $\gd_{\calF_2}(\Gamma)\le 2$, for this  we construct a model for $E_{\calF_2}\Gamma$ of dimension 2. The construction of such a model follows the same strategy as in the proof of \cref{thm:case:plane} c). We include the details for the sake of completeness.

By \cref{lemma:case:sol}, $\calF_2=\calF_1\cup \calF(K_1,K_2)$, where $\calF(K_1,K_2)$ is the smallest family of $\Gamma$ containing $K_1$ and $K_2$. Then, by \cite[Lemma 4.4]{DQR11} the following  homotopy $\Gamma$-push-out give a model for $E_{\calF_2}\Gamma$
 \begin{equation}\label{c350}
 \xymatrix{E_{\calF_1(K_1,K_2)}\Gamma\ar[d]^g \ar[r]^f  & E_{\calF_1}\Gamma\ar[d]\\
E_{\calF(K_1,K_2)}\Gamma \ar[r] & X
}
\end{equation}
where $\calF_1(K_1,K_2)=\calF(K_1,K_2)\cap \calF_1$. We claim that, with suitable choices, $X$ is 2-dimensional. Let us explain first the idea of the construction. First note that $H$ is a normal subgroup of $\Gamma$, and $\Gamma/H\cong \Z_2\ast \Z_2$. Let $\calF(\dbZ_2, \dbZ_2)$ be the smallest family of $\Z_2\ast \Z_2$ containing both $\Z_2$ factors, and notice that $\calF(K_1,K_2)$ coincides with the smallest family of $\Gamma$ containing all the preimages of elements in $\calF(\Z_2,\Z_2)$ under the projection $\Gamma\to \Z_2\ast \Z_2$. Recall that $\Z_2\ast \Z_2$ acts on $\R$ by simplicial isometries, in fact, the first factor acts as a reflection through 0, and the second factor as a reflection through 1/2, and with this action $E_{\calF(\dbZ_2, \dbZ_2)}(\dbZ_2*\dbZ_2)=\R$.  Now it is easy to verify that $E_{\calF(\dbZ_2, \dbZ_2)}(\Gamma/H)=E_{\calF(\dbZ_2, \dbZ_2)}(\dbZ_2*\dbZ_2)=\R$ is a model for $E_{\calF(K_1,K_2)}\Gamma$.

 Next, we will show that there is a 3-dimensional model $Y$ for $E_{\calF_1}\Gamma$ such that 
\begin{enumerate}
    \item $Y$ is the union of two $\Gamma$-subcomplexes $Y_1$ and $Y_2$.
    
    \item every 3-cell in $Y$ belongs to $Y_2$.
    
    \item $Y_2$ is model for $E_{\calF_1(K_1,K_2)}\Gamma$.
    \item The map $f$ is the inclusion $Y_2\to Y$.
\end{enumerate}
Since $f$ is an inclusion, then the homotopy $\Gamma$-push-out \eqref{c350} can be replaced by an honest $\Gamma$-pushout, see for example \cite[Theorem~1.1]{Wa80}. That is, we take $X=Y\cup_g \mathbb R=Y\sqcup \mathbb R/\sim$, where we identify $x\sim g(x)$ for all $x\in Y_2$. By (1)-(4), we can see that every 3-cell of $Y$ is being collapsed to a 1-dimensional space via $g$, hence we get that $X$ is a complex of dimension less than or equal to 2. From the explicit construction we actually conclude $X$ is of dimension 2.
Let us construct $Y$, $Y_1$ and $Y_2$.  Since $\Gamma$ is a torsion-free and virtually poly-$\Z$ group (it contains an index two subgroup isomorphic to $\Z^2\rtimes \Z)$, by  \cite[Lemma 5.15 and Theorem 2.3]{luck:Weiermann} we get a model $Y$ for $E_{\calF_1} \Gamma$ via the following $\Gamma$-push-out 
\begin{equation}\label{cdiagram}
 \xymatrix{\displaystyle\bigsqcup_{C\in \I}\Gamma\times_{N_\Gamma C} EN_{\Gamma}C\ar[d] \ar[r]  & E\Gamma\ar[d]\\
\displaystyle\bigsqcup_{C\in \I}\Gamma\times_{N_\Gamma C} E W_{\Gamma}C \ar[r] & Y
}
\end{equation}
where  $\I$ is a set of   representatives of commensuration classes in $\calF_1-\calF_0$. To construct the aforementioned $\Gamma$-pushout we use a construction analogue to the one described in \cref{conditions:exist:pushout}, that is we replace $E\Gamma$ with the mapping cylinder of the top horizontal arrow.

Since $I$ is a set of representatives of commensuration classes, we can assume that every $C\in I$ is a subgroup of the index two subgroup $\Z^2\rtimes \Z$ of $\Gamma$. Now by \cref{lemma:case:sol} b) we have that $I$ is disjoint union of $\I_1=\{C\in \I | N_{\Gamma}C\text{ is virtually }\Z \}=\{C\in \I | C\not\le H \}$ and $\I_2=\{C\in \I | N_{\Gamma}C \text{ is virtually }\Z^2\}=\{C\in \I | C\le H \}$. Then the $\Gamma$-push-out \eqref{cdiagram} can written as  

\begin{equation}\label{cpushout:space:classifying:F1}
 \xymatrix{\left(\displaystyle\bigsqcup_{C\in \I_1}\Gamma\times_{N_\Gamma C} EN_{\Gamma}C\right)\displaystyle\bigsqcup\left(\displaystyle\bigsqcup_{C\in \I_2}\Gamma\times_{N_\Gamma C} EN_{\Gamma}C\right)\ar[d] \ar[r]  & E\Gamma\ar[d]\\
\left(\displaystyle\bigsqcup_{C\in \I_1}\Gamma\times_{N_\Gamma C} EW_{\Gamma}C\right) \displaystyle\bigsqcup\left(\displaystyle\bigsqcup_{C\in \I_2}\Gamma\times_{N_\Gamma C} EW_{\Gamma}C\right)  \ar[r] & Y
}
\end{equation}

Define $Y_1$ and $Y_2$ via the following $\Gamma$-push-outs

\begin{equation*}
 \xymatrix{\displaystyle\bigsqcup_{C\in I_1}\Gamma\times_{N_\Gamma C} EN_{\Gamma}C\ar[d] \ar[r]  & E\Gamma\ar[d]\\
\displaystyle\bigsqcup_{C\in I_1}\Gamma\times_{N_\Gamma C} EW_{\Gamma}C  \ar[r] & Y_1 
}\quad\quad\quad\quad\quad
\xymatrix{\displaystyle\bigsqcup_{C\in I_2}\Gamma\times_{N_\Gamma C} EN_{\Gamma}C\ar[d] \ar[r]  & E\Gamma\ar[d]\\
\displaystyle\bigsqcup_{C\in I_2}\Gamma\times_{N_\Gamma C} EW_{\Gamma}C  \ar[r] & Y_2 
}
\end{equation*}
We used the cylinder construction described in  \cref{conditions:exist:pushout} to construct all the pushouts above.
Clearly $Y=Y_1\cup Y_2$, and $Y_1\cap Y_2=E\Gamma$. Since $\calF_1(K_1,K_2)$ is the smallest family of $\Gamma$ containing the  virtually cyclic subgroups of $K_1$ and $K_2$, a direct application of  \cite[Lemma 5.15 and Theorem 2.3]{luck:Weiermann} leads to the fact that $Y_2$ is a model for $E_{\calF_1(K_1,K_2)}\Gamma$. 

 We note the following:
\begin{enumerate}[i)]
    \item If $C\in \I_1$, then $N_\Gamma C$ is virtually $\Z$, and by thus by \cite[Theorem~5.13]{luck:Weiermann}   $EN_{\Gamma}C$ has a 1-dimensional model. Thus the cylinder associated to $EN_{\Gamma}C$  contributes to $E_{\calF_1}\Gamma$ with a subspace of dimension 2. 
    \item If $C\in \I_2$, then $N_\Gamma C$ is a torsion-free virtually $\Z^2$ group, thus by \cite[Theorem~5.13]{luck:Weiermann} $EN_{\Gamma}C$ has a 2-dimensional model. Thus the cylinder associated to $EN_{\Gamma}C$ contributes to $E_{\calF_1}\Gamma$ with a subspace of dimension 3.
    \item By \cite[Theorem 5.13]{luck:Weiermann} $E\Gamma$ has a model of dimension 3.
\end{enumerate}  

As a consequence of the above observations, we conclude that every 3-cell of $Y$ belongs to $Y_2$. This concludes the construction of a 2-dimensional model for $E_{\calF_2}\Gamma$.
\end{proof}

\section{A small detour on acylindricity of groups acting on trees}\label{section:acylindricity}

In this section we prove \cref{acylindrical:stabilizer:geodesic} and \cref{building:model:apply:haefliger}. These results will be relevant in the next section to compute the $\calF_2$-dimension of fundamental group of prime manifolds from the $\calF_2$-dimension of the JSJ-pieces, and to compute the $\calF_2$-dimension of a 3-manifold group from he $\calF_2$-dimensions of the prime pieces.

In the following definition we state the notion of acylindricity, which will be key to understand the abelian subgroups of a 3-manifold group.

\begin{definition}
Let $Y$ be a graph of groups with fundamental group $G$. The splitting  $G=\pi_1(Y)$ is \emph{acylindrical} if there is a integer  $k$ such that, for every path  $\gamma$ of length  $k$ in the Bass-Serre tree $T$ of  $Y$,  the stabilizer of $\gamma$ is finite. 
\end{definition}

Recall a \emph{geodesic line} of a simplicial tree $T$, is a simplicial embedding of $\mathbb{R}$ in $T$, where $\mathbb{R}$ has as vertex set $\dbZ$ and an edge joining any two consecutive integers.

\begin{lemma}\label{acylindrical:stabilizer:geodesic}
 Let  $Y$ be a graph of groups with fundamental group   $\Gamma$, and Bass-Serre tree  $T$. Suppose that the splitting of  $\Gamma$ is acylindrical. Then
 \begin{enumerate}[a)]
     \item the setwise stabilizer of every geodesic line  in  $T$ is virtually cyclic.
     
     \item every  virtually $\Z^n$ subgroup of $\Gamma$ with $n\ge 2$  fixes a vertex of  $T$.
 \end{enumerate} 
\end{lemma}

\begin{proof}
a) Let  $L$ be a geodesic line of $T$. 
Denote by  $\fix_{\Gamma}(L)$ (resp. $\stab_{\Gamma}(L)$) the subgroup of all elements of  $\Gamma$ that fix $L$ pointwise (resp. setwise). Consider $L$ with simplicial structure induced by $T$. Note that the group of simplicial automorphisms of $L$, denoted $\aut(L)$, is isomorphic to the infinite dihedral group $D_{\infty}$. Since $\stab_{\Gamma}(L)$ acts by simplicial automorphisms on $L$, we have a homomophism $\stab_{\Gamma}(L)\to \aut(L)=D_\infty$ with kernel $\fix_{\Gamma}(L)$ and image a certain subgroup $D$ of $D_\infty$. In other words we have a short exact sequence  
$$1\to \fix_{\Gamma}(L) \to \stab_{\Gamma}(L)\to D \to 1.$$
Since $L$ contains arbitrarily long paths of $T$, the acilyndricity hypothesis implies that $\fix_{\Gamma}(L)$ is finite. On the other hand, since $D_\infty$ if a virtually cyclic group, then $D$ also is virtually cyclic. Therefore  $\stab_{\Gamma}(L)$ is virtually cyclic. 
 
b) Let  $H$ be a virtually $\Z^n$ subgroup of  $\Gamma$ with  $n\ge 2$, then by  \cref{finitely:generated:group:fix:edge:or:act:path} it happens exactly one of the following: either $H$ fixes a vertex of $T$ or $H$ acts co-compactly in a unique geodesic line  $L$ of  $T$. It remains to rule out the second possibility. Suppose that  $H$ acts co-compactly in a unique geodesic line $L$ of $T$.
Let $S$ be a finite index subgroup of  $H$ isomorphic to $\Z^n$. Note that  $S$ acts by restriction on   $L$, i.e. we have a homomorphism  $S\xrightarrow{\varphi} \Aut(L)$ whose image contains an infinite subgroup of translations. Since $\Aut(L)\cong D_{\infty}$, we have $\varphi(S)$ either is isomorphic to $\Z$ or to   $D_{\infty}$. But  $\varphi(S)$ is not isomorphic to $D_{\infty}$, otherwise the abelian group $S/\ker(\varphi)$ would be isomorphic to the non-abelian group $D_{\infty}$. The homomorphism   $S\xrightarrow{\varphi} \Aut(L)$ induces the following short exact sequence  
 $$1\to \ker(\varphi) \to S\to \varphi(S)\cong\dbZ\to 1.$$
Note that every element of $\ker(\varphi)$, by definition, acts trivially on $L$, thus $\ker(\varphi)\subseteq\fix_\Gamma(L)$. On the other hand, since $L$ is an (infinite) geodesic of $T$, it contains paths of arbitrary large length, therefore $\fix_\Gamma(L)$ fixes  arbitrarily long paths. The acilydricity hypothesis implies that $\fix_\Gamma(L)$ is finite. Hence $\ker(\varphi)$ is a finite subgroup of the torsion-free group $S$, thus $\ker(\varphi)$ is trivial. We conclude that $S\cong \dbZ^n$ embeds in $\dbZ$ which is a contradiction.
\end{proof}

In the following theorem we describe a 2-dimensional model $\widetilde T$ for a classifying space with respect to a family that contains the family $\calF_n$. This $\widetilde T$ will be used in the next section to compute the $\calF_2$ dimension of a 3-manifold group using the prime splitting and the JSJ decomposition.

\begin{theorem}\label{building:model:apply:haefliger}
Let $Y$ be a graph of groups with fundamental group  $\Gamma$ finitely generated  and Bass-Serre tree  $T$. Consider the collection  $\mathcal{A}$ of all the geodesics of   $T$  that admit a  co-compact action of an infinite  virtually  cyclic subgroup of $\Gamma$. Then the space $\widetilde{T}$ given by the following  homotopy $\Gamma$-push-out

\begin{equation}
 \xymatrix{\displaystyle\bigsqcup_{\gamma\in \mathcal{A}}\gamma\ar[d] \ar[r]  & T\ar[d]\\
\displaystyle\bigsqcup_{\gamma\in \mathcal{A}} \{*_{\gamma}\}  \ar[r] & \widetilde{T} 
}
\end{equation}
is a model $\widetilde{T}$ for $E_{\iso_{\Gamma}(\widetilde{T})}\Gamma$ where  $\iso_{\Gamma}(\widetilde{T})$ is the family generated by the isotropy groups of $\widetilde{T}$, i.e. by coning-off on $T$ the geodesics in $\mathcal{A}$  we obtain a model for  $E_{\iso_{\Gamma}(\widetilde{T})}\Gamma$. Moreover,  if  the splitting  $\Gamma=\pi_1(Y)$ is acylindrical, then the family  $\iso_{\Gamma}(\widetilde{T})$ contains the family  $\calF_n$ of $\Gamma$ for all $n\ge 0$.  
\end{theorem}

\begin{proof} If $H\leq \Gamma$ acts co-compactly on the geodesic line $\gamma$ of $T$ and $g\in \Gamma$, then $gHg^{-1}$ acts co-compactly on $g\gamma$. It follows that both $\bigsqcup_{\gamma\in \mathcal{A}}\gamma$ and $\bigsqcup_{\gamma\in \mathcal{A}}\{*_\gamma\}$ are $\Gamma$-CW-complexes, and therefore the space  $\widetilde{T}$ is a   $\Gamma$-CW-complex. 

Clearly $\widetilde{T}^{K}$ is nonempty if an only if  $K\in \iso_{\Gamma}(\widetilde{T})$. We will see that for  $K\in \iso_{\Gamma}(\widetilde{T})$ the fixed point set $\widetilde{T}^K$ is contractible.  We have two cases: $T^{K}\neq\emptyset$ or $T^{K}=\emptyset$. In the first case we have that $T^K$ is a sub-tree of $T$. Thus $\widetilde{T}^K$ is obtained from $T^K$ by coning-off some geodesic segments, then the space  $\widetilde{T}^K$ which is contractible, it follows that $\widetilde{T}^K$ is contractible. In the second case we have that  $\widetilde{T}^K$ consists of a union of some cone points. Note that $*_{\gamma}\in \widetilde{T}^K$ if and only if  $K\le \stab_{\Gamma}(\gamma)$. By hypothesis $T^K=\emptyset$, then from   \cite[Corollary~3]{Serre:trees} we have there is a hyperbolic element $h\in K$ that acts co-compactly on $\gamma$. Since $h$ acts co-compactly on a unique geodesic of $T$, we conclude $\widetilde{T}^K=*_{\gamma}$, therefore it is contractible.

We show that the family $\calF_n$ of $\Gamma$ is contained in  $\iso_{\Gamma}(\widetilde{T})$. Let $K\in \calF_n$ then we have three cases: $K$ is finite, $K$ is virtually  $\Z$ or $K$ is virtually $\Z^k$ with $k\ge 2$. If $K$ is finite, then it is well-known that $K$ has a fixed  point  in $ T$, thus  $K\in \iso_{\Gamma}(\widetilde{T})$. If $T$ is virtually  $\Z$ then, by  \cref{finitely:generated:group:fix:edge:or:act:path}, $K$ fixes a vertex in $T$ or it acts co-compactly on a unique geodesic  $\gamma_H$. In the first case $K$ it is clear that  $K\in \iso_{\Gamma}(\widetilde{T})$, while for the second case $K$ fixes $*_\gamma$, and therefore  $K\in \iso_{\Gamma}(\widetilde{T})$. Finally, if  $K$ is virtually $\Z^k$ with  $k\ge 2$, we have by  \cref{acylindrical:stabilizer:geodesic}, that   $K$ fixes a point in $ T$, and therefore $K\in \iso_{\Gamma}(\widetilde{T})$.

\end{proof}

\section{Proofs of \cref{thm:principal:1} and \cref{geometric:dimension:prime}}\label{section:final}

In this section we prove the main theorems of the present paper, but before we need some preliminary results.

\begin{proposition}\label{prop:F3:equal:Fk}
Let $\Gamma$ be fundamental group of a 3-manifold $M$ that it is either hyperbolic, Seifert fiber possibly with non-empty boundary, or modelled on $\sol$. The following statements hold 
\begin{enumerate}
    \item If $M$ is not modelled on $\mathbb E^3$, then $\calF_2=\calF_k$ for all $k\geq 3$. In particular $\gd_{\calF_2}(\Gamma)=\gd_{\calF_k}(\Gamma)$ for all $k\geq 3$.
    
    \item If $M$ is modelled on $\mathbb E^3$, then $\calF_3=\calF_k$ for all $k\geq 4$. Moreover, $\gd_{\calF_k}(\Gamma)=0$ for all $k\geq 3$.
\end{enumerate}
\end{proposition}
\begin{proof}
By \cite[Proposition~A]{onorio}, $\gd_{\calF_2}(\Z^3)=5$. As a consequence, if $\Gamma$ contains a subgroup isomorphic to $\Z^3$, then $\gd_{\calF_2}(\Gamma)\geq 5$. By the second column of \cref{summary:dimention:pieces:jsj:Fk} and \cref{dg:manifold:prime:sol} we conclude that $\Gamma$ does not contain a $\Z^3$-subgroup unless $M$ is modelled on $\dbE^3$. This proves the first item. The second item follows by noticing that if $M$ is modelled on $\dbE^3$, then $\Gamma$ is virtually $\dbZ^3$.
\end{proof}

\begin{corollary}\label{coro:Zn:subgroups}
Let $\Gamma$ be the fundamental group of the 3-manifold $M$.
Let $H$ be a $\Z^n$-subgroup of $\Gamma$, then $n\leq 3$. Moreover $\Gamma$ contains a $\Z^3$-subgroup if and only if one of the prime pieces of $M$ is modelled on $\dbE^3$.
\end{corollary}
\begin{proof}
Let $n\geq2$.
Let $\Gamma=\Gamma_1\ast \cdots\ast \Gamma_r$ be the splitting of $\Gamma$ associated to the prime decomposition of $M$, and let $H\leq\Gamma$ be a $\dbZ^n$-subgroup of $\Gamma$. Then by Kurosh subgroup theorem $H$, without loss of generality, is a subgroup of $\Gamma_1$. Next we look at the graph of groups $Y$ given by the JSJ decomposition of $\Gamma_1$, in particular the vertex groups of $Y$ are the fundamental groups of the JSJ pieces of $\Gamma_1$ and the edge groups are isomorphic to the fundamental groups of the tori in the JSJ decomposition. Then by \cref{graphgroup:jsj:prime:manifold:acylindrical} the splitting $\Gamma_1=\pi_1(Y)$ is acylindrical, and by \cref{acylindrical:stabilizer:geodesic}, $H$ fixes a vertex of the Bass-Serre tree of $Y$, thus $H$ is conjugated to a subgroup of the fundamental group of a JSJ piece $N$ of $\Gamma_1$. By \cref{prop:F3:equal:Fk} we conclude $n\leq 3$. Moreover, if we assume $n=3$,  by \cref{prop:F3:equal:Fk} such a piece must be modelled on $\dbE^3$. On the other hand every manifold modelled on $\dbE^3$ has empty boundary, thus $N$ is a prime piece of $\Gamma$. Finally, if $M$ has a prime piece modelled on $\dbE^3$ it is clear that $\Gamma$ contains a $\Z^3$-subgroup. 
\end{proof}

\begin{theorem}\cite[Proposition 8.2]{Luis:Lafon:Joeken}\label{graphgroup:jsj:prime:manifold:acylindrical}
Let  $M$ be a closed, oriented, connected, prime 3-manifold which is not geometric.  Let $Y$ be the graph of groups associated to its minimal JSJ decomposition. Then the splitting of  $G=\pi_1(Y)$ as the fundamental group of $Y$ is acylindrical.  \end{theorem}

\begin{proposition}\label{geometry:dimension:approximation:manifold:prime}
Let $M$ be a closed, oriented, connected, prime 3-manifold   which is not geometric.
Let  $N_1,N_2,\dots,N_r$ be the pieces of the minimal JSJ decomposition of  $M$. Denote  $\Gamma=\pi_1(M)$ and  $\Gamma_i=\pi_1(N_i)$. If $k\ge 2$, then  \[\max\{\gd_{\calF_k}(\Gamma_i)|1\le i\le r\}\leq \gd_{\calF_k}(\Gamma)\le \max\{2,\gd_{\calF_k}(\Gamma_i)|1\le i\le r\}.\]
\end{proposition}
\begin{proof}
For all  $1\le i\le r$, the group $\Gamma_i$ is a subgroup of $\Gamma$, then  $\gd_{\calF_k}(\Gamma_i)\le \gd_{\calF_k}(\Gamma)$, and the first inequality follows.

Now we show the second inequality. Let  $Y$ the graph of groups associated to the JSJ decomposition of $M$  with Bass-Serre tree  $T$. By  \cref{graphgroup:jsj:prime:manifold:acylindrical} the descomposition of   $\pi_1(M)=\Gamma$ is acylindrical. Therefore we can use  \cref{building:model:apply:haefliger} to obtain a 2-dimensional space $\widetilde{T}$  what is obtained from $T$ coning-off some geodesics of $T$, the space  $\widetilde{T}$  is a model for  $E_{\iso_{\Gamma}(\widetilde{T})}\Gamma$ where $\iso_{\Gamma}(\widetilde{T})$ is the family generated by the isotropy groups of  $\widetilde{T}$, and $\calF_k\subset \iso_{\Gamma}(\widetilde{T})$. We have everything set up to apply \cref{prop:haefliger}, that is we only have to compute $\gd_{\calF_k\cap \Gamma_\sigma}(\Gamma_\sigma)+\dim(\sigma)$ for each cell $\sigma$ of $\widetilde T$. Once done this the proof will be complete.
\begin{itemize}
    \item If the  0-cell  $\sigma$ of $\widetilde T$ belongs to  $T$, then  $\Gamma_{\sigma}=\Gamma_i$ for some   $1\le i\le r$. If the  0-cell $\sigma$ belongs to $\widetilde{T}-T$, then $\Gamma_{\sigma}$ is is the setwise stabilizer of a geodesic of $T$, and therefore is virtually cyclic by \cref{acylindrical:stabilizer:geodesic}. Hence in this case $\gd_{\calF_k\cap \Gamma_\sigma}(\Gamma_\sigma)+\dim(\sigma)=\gd_{\calF_k}(\Gamma_i)$ or 0.
    \item If the  1-cell  $\sigma$ of $\widetilde T$ belongs to  $T$, then  $\Gamma_{\sigma}$ is isomorphic to $\Z^2$. If the  1-cell $\sigma$ has a vertex in  $\widetilde{T}- T$, then  $\Gamma_{\sigma}$ is virtually cyclic as in the previous item. Hence in this case $\gd_{\calF_k\cap \Gamma_\sigma}(\Gamma_\sigma)+\dim(\sigma)=1$.
    \item If $\sigma$ is a 2-cell of $\widetilde{T}$, then it always contain a vertex of $\widetilde{T}- T$, and therefore $\Gamma_\sigma$ is virtually cyclic. Hence in this case $\gd_{\calF_k\cap \Gamma_\sigma}(\Gamma_\sigma)+\dim(\sigma)=2$.
\end{itemize}
\end{proof}

\begin{proof}[Proof of \cref{geometric:dimension:prime}] If the minimal JSJ decomposition of  $M$ has only one piece, then the manifold cannot be modelled on $\sol$, since such manifolds are neither Seifert nor  hyperbolic. Hence if $M$ has only one JSJ piece the theorem  follows. From now on, suppose that the minimal JSJ decomposition of $M$ has more than one piece. We have two cases: $M$ is geometric or  not.
If $M$ is not geometric we claim that $$ \gd_{\calF_k}(\Gamma)= \max\{\gd_{\calF_k}(\Gamma_i)|1\le i\le r\}.$$
By \cref{geometry:dimension:approximation:manifold:prime} it is enough to see that there is a piece  $N_i$ in the minimal JSJ decomposition of $M$ such that  $\gd_{\calF_k}(\Gamma_i)\ge 2$. By definition, the pieces in the JSJ decomposition of  $M$  are either hyperbolic with boundary or Seifert fiber with boundary.
If the JSJ decomposition of $M$ has a hyperbolic piece or a Seifert fiber with base orbifold $B$ modelled on $\mathbb{H}^2$, then we are done, since  by the Table  \ref{summary:dimention:pieces:jsj:Fk} the fundamental groups of these pieces have $\gd_{\calF_k}$ equal to  3 and  2 respectively. It remains to see what happens if we only we have  Seifert fiber pieces with base orbifold $B$ modelled on $\E^2$. By  \cref{thm:seifer:base:flat:bondary} these pieces either are diffeomorphic to $T^2\times I$ or to   twisted $I$-bundle over the Klein bottle. If we have a piece of the form $T^2\times I$, then minimal the JSJ decomposition of  $M$  would have only one  piece, otherwise we will contradict the minimality of the  JSJ decomposition. Then $M$ would be the mapping torus of a self-diffeomorphism of $T^2$ and, by \cite[Theorem 1.10.1., p.23]{matthias:stefan:henry}, $M$ would be geometric. Since we are in the non-geometric case we discard this possibility. By last, we see that happen if we only have pieces homeomorphic to twisted  I-bundle over the Klein bottle. Note that these pieces only have only one boundary component, therefore we only can have two such pieces glued by a diffeomorphism between their boundaries. In   \cite[p.19, last paragraph]{matthias:stefan:henry} they show that $M$ is geometric, and once more we discard this possibility.

Suppose now that $M$ is geometric with at least two JSJ pieces. Note that $M$ is not hyperbolic, otherwise the JSJ decomposition of $M$ would have  only one piece. If  $M$ is modelled on $\sol$ the theorem follows from  \cref{dg:manifold:prime:sol}. Finally, if  $M$ is not hyperbolic nor modelled on  $\mathrm{Sol}$, then by  \cite[Theorem 1.8.1, p.17]{matthias:stefan:henry} $M$ is  Seifert fiber. But this cannot happen since this implies we have only one JSJ piece.
\end{proof}

\begin{theorem}\label{approximation:dimention:manifold}
Let  $M$ be a closed, connected, oriented 3-manifold.
Let  $P_1,P_2,\dots,P_r$ be the pieces of the prime decomposition of $M$.  Denote  $\Gamma=\pi_1(M)$ and  $\Gamma_i=\pi_1(P_i)$. If $k\ge 2$, then \[\max\{\gd_{\calF_k}(\Gamma_i)|1\le i\le r\}\leq \gd_{\calF_k}(\Gamma)\le \max\{2,\gd_{\calF_k}(\Gamma_i)|1\le i\le r\}.\]
\end{theorem}
\begin{proof}
Let $T$ be the Bass-Serre tree of the splitting $\Gamma=\Gamma_1\ast \cdots \ast \Gamma_r$. Since the edge stabilizers of $T$ are trivial, the splitting of $\Gamma$ is acylindrical. Now the proof is completely analogous to the proof of \cref{geometry:dimension:approximation:manifold:prime} and the details are left to the reader.
\end{proof}

\begin{lemma}\label{free:product:subgroup:free:no:cyclic}
Let $G=H_1* \dots *H_k$ with $k\ge 2$ and $H_i\neq 1$ for all  $i$. Then exactly one of the following hold:
\begin{enumerate}[a)]
    \item $G$ is isomorphic to  $D_{\infty}$ with  $k=2$ and $H_1, H_2$ isomorphic to   $\Z_2$ or
    \item $G$ contains a non-cyclic free subgroup.
    
    \end{enumerate}
\end{lemma}
\begin{proof}
Consider  $G$ as the following two fold free product $H_1*(H_2*\cdots*H_k)$. By \cite[Lemma 1.11.2, p.24]{matthias:stefan:henry} this free product contains a non-cyclic free subgroup unless  the two factors are isomorphic to   $\Z_2$. By hypothesis  $H_i\neq 1$ for all $i$, then   $H_2*\cdots*H_k$ is asomorphic to   $\Z_2$ if and only if it has only one factor isomorphic to  $\Z_2$. Therefore the factors of $G=H_1*(H_2*\cdots*H_k)$ are isomorphic to  $\Z_2$ if and only if  $k=2$ and $H_1,H_2$ are isomorphic to $\Z_2$. \end{proof}

\begin{proof}[Proof of \cref{thm:principal:1}]
If we  only have one piece the theorem follows as we are necessarily in the third case of our conclusion. From now on suppose that we have at least two pieces in the prime decomposition. Then   $\Gamma=\pi_1(P_1)\ast \pi_1(P_2)\ast \cdots \ast \pi_1(P_r)$ with $r\ge 2$. By  \cref{free:product:subgroup:free:no:cyclic} we have two cases: the group  $\Gamma$ is somorphic to $D_{\infty}$ with  $r=2$ and $\pi_{1}(P_1), \pi_{1}(P_2)$ are isomorphic to $\Z_2$, or $\Gamma$ contains a non-cyclic free subgroup.  In the first case we are done, since $\pi_1(P_1)=\Z_2$ implies that $P_1$ is homeomorphic to $\dbR P^3$.

From now on suppose also that $\Gamma$ contains a non-cyclic free subgroup. Then $\Gamma$ is not virtually abelian, and by \cref{lemma:low:dimensions} we get $\gd_{\calF_k}(\Gamma)\geq 2$. Next we consider two cases for fix a $k$: $\gd_{\calF_k}(\Gamma_i)=0$ for all $i$, or not.  In the first case, by \cref{approximation:dimention:manifold}, $\gd_{\calF_k}(\Gamma)\leq 2$, and therefore $\gd_{\calF_k}(\Gamma)= 2$, hence we are done in this case. In the second case we have that there is a $\Gamma_s$ such that $\gd_{\calF_k}(\Gamma_s)\neq 0$, and by \cref{lemma:low:dimensions}, we have $\gd_{\calF_k}(\Gamma_s)\geq 2$. Therefore by \cref{approximation:dimention:manifold},
\[\gd_{\calF_k}(\Gamma)=\max\{\gd_{\calF_k}(\Gamma_i)|1\le i\le r\},\]
and we are done in this final case.

To finish, we prove the moreover part of the statement. Assume $\Gamma=\pi_1(M)$ is virtually abelian. By \cite[Theorem~1.11.1]{matthias:stefan:henry}, $M$ is spherical, $\mathbb RP^3\# \mathbb RP^3$, $\mbs^1\times \mbs^2$, or is covered by a torus bundle. If $\Gamma$ is covered by a torus bundle, then it has a finite index subgroup isomorphic to $K=\dbZ^2\rtimes_\varphi \dbZ$, and therefore $K$ must be also virtually abelian. Since $K$ is poly-$\dbZ$ of rank 3, then $K$ must be virtually $\dbZ^3$, thus $K$ and $\Gamma$ are modelled on $\dbE^3$. Now the moreover part follows easily.
\end{proof}

\bibliographystyle{alpha} 
\bibliography{myblib}
\end{document}